\documentclass[11pt]{amsart}

\usepackage{amsmath}
\usepackage{amstext}
\usepackage{amssymb}
\usepackage{amsthm}
\usepackage{enumerate}
\usepackage[all]{xy}
\usepackage{mathrsfs}
\usepackage{datetime}
\newcommand{\inx}[1]{\ \epsilon_{#1}\ }

\makeatletter
\def\imod#1{\allowbreak\mkern10mu({\operator@font mod}\,\,#1)}
\makeatother

\newtheorem{theorem}{Theorem}[section]

\newtheorem{lemma}[theorem]{Lemma}

\newtheorem{corollary}[theorem]{Corollary}
\newtheorem{conjecture}[theorem]{Conjecture}

\theoremstyle{definition}
\newtheorem{definition}[theorem]{Definition}
\newtheorem{problem}[theorem]{Problem}

\theoremstyle{remark}
\newtheorem{remark}[theorem]{Remark}

\theoremstyle{remark}

\numberwithin{equation}{section}

     \DeclareMathOperator{\Aut}{Aut}

    \DeclareMathOperator{\length}{length}
    \DeclareMathOperator{\dom}{dom}

    \DeclareMathOperator{\proj}{proj}

    \DeclareMathOperator{\Char}{Char}
     \DeclareMathOperator{\II}{II}
    \DeclareMathOperator{\vN}{vN}\DeclareMathOperator{\SL}{SL}
    \DeclareMathOperator{\III}{III}

    \DeclareMathOperator{\stab}{stab}
    \DeclareMathOperator{\Coll}{Coll}

    \newcommand{\restrict}{\!\upharpoonright\!}
    \newcommand{\vrestrict}{\upharpoonright}

    \newcommand{\act}{A}
    \newcommand{\actfwm}{A^{\star}_{\rm wm}}
    \newcommand{\actson}{\curvearrowright}
    
    \newcommand{\forces}{\Vdash}

    \newcommand{\conj}{\simeq_{\rm C}}
    \newcommand{\oeq}{\simeq_{\rm OE}}
    \newcommand{\vneq}{\simeq_{\rm vNE}}
    \newcommand{\psX}{\mathscr X}
    \newcommand{\psY}{\mathscr Y}
    \newcommand{\psZ}{\mathscr Z}

\def\C{{\mathbb C}}

\def\N{{\mathbb N}}
\def\Z{{\mathbb Z}}

\def\F{{\mathbb F}}
\def\P{{\mathbb P}}
\def\T{{\mathbb T}}

\def\sss{\scriptscriptstyle}

\title[The Borel complexity of von Neumann equivalence]{The Borel complexity of von Neumann equivalence}

\author{Inessa Epstein}

\author{Asger T\"ornquist}



\date{\today}

\begin{document}

\maketitle

\begin{center} {\it \center Dedicated to the memory of Greg Hjorth (June 14, 1963 -- January 13, 2011)}\end{center}

\begin{abstract}
We prove that for a countable discrete group $\Gamma$ containing a copy of the free group $\F_n$, for some $2\leq n\leq\infty$, as a normal subgroup, the equivalence relations of conjugacy, orbit equivalence and von Neumann equivalence of the ergodic a.e. free probability measure preserving actions of $\Gamma$ are analytic non-Borel equivalence relations in the Polish space of probability measure preserving $\Gamma$ actions. As a consequence we obtain that the isomorphism relation in the spaces of separably acting factors of type $\II_1$, $\II_\infty$ and $\III_\lambda$, $0\leq\lambda\leq 1$, are analytic and not Borel when these spaces are given the Effros Borel structure.
\end{abstract}

\section{Introduction}

A fundamental problem of ergodic theory is the \emph{conjugacy problem}: Given two measure preserving actions of a countable discrete group $\Gamma$ on a standard probability space, how does one determine if they are conjugate actions? A solution to the conjugacy problem should ideally be a method which, when given two measure preserving actions, can be applied systematically, and will produce a yes-or-no answer to the conjugacy question.

In those cases where we have nice classification theorems for the probability measure preserving (p.m.p.) actions of $\Gamma$, then the classification also solves the conjugacy problem: For instance, when $\Gamma=\Z$ (or, more generally is amenable, \cite{orwe87}, or even is a non-amenable free group, \cite{bowen10b}) and we consider only Bernoulli actions of $\Gamma$, then the conjugacy problem can be solved by computing the entropy of the two actions, which by a celebrated Theorem of Ornstein \cite{ornstein70} is a complete invariant for conjugacy. However, the conjugacy problem may be viewed as distinct from the classification problem: Having a method for answering the yes-or-no question of conjugacy clearly does not provide a classification, and a classification may assign invariants for which it is difficult to determine if the assigned invariants are isomorphic or not.

The conjugacy problem arguably goes back to Halmos, who posed it in the context of $\Gamma=\Z$ as Problem 3 in \cite[p. 96]{halmos56}. As stated above the conjugacy problem is vague since it is not clear what is meant by a ``method''. One possible way of posing the conjugacy problem in a precise mathematical way is the following:

\begin{problem}[Kechris {\cite[18.(IVb)]{kechris10}}]\label{pr.kechris}
Is the conjugacy relation for p.m.p. (ergodic, a.e. free) actions of a countable discrete group $\Gamma$ a Borel or analytic set? If it is analytic, is it complete analytic?
\end{problem}

In \S 2 below we will give an explanation for why this question is closely related to the conjugacy problem. Roughly speaking, if the conjugacy relation were Borel, then the description of the Borel set (that is, how it is build up using countable unions and complements) would provide a method for determining if two actions are conjugate, and this method would use only countable resources. If, on the other hand, the conjugacy relation is analytic and not Borel, then no generally applicable method that relies on only countable resources could solve the conjugacy problem; and in case the conjugacy relation is complete analytic, then the \emph{worst possible} general method (which we describe in \S 2) would also be the \emph{best possible}.

When $\Gamma=\Z$, the conjugacy problem was solved by Hjorth in \cite{hjorth01}, however the solution suffered from the obvious defect that it only showed that the conjugacy relation on \emph{non-ergodic} measure preserving transformations is analytic and not Borel. Only recently was the conjugacy problem for $\Z$-actions given a satisfactory solution: In \cite{foruwe11}, Foreman, Rudolph and Weiss showed that the conjugacy relation on \emph{ergodic} actions of $\Z$ is a complete analytic set (see also \cite{foruwe06}.) In this paper we will will prove the following:

\begin{theorem}\label{t.mainintro1}
Let $\Gamma$ be a countable discrete group containing a non-amenable free group $\F_n$, $2\leq n\leq\infty$, as a normal subgroup. Then the conjugacy relation for weakly mixing a.e. free p.m.p. actions of $\Gamma$ is complete analytic, and so it is not Borel.
\end{theorem}

This in particular settles the conjugacy problem form $\Gamma=\F_n$ for $2\leq n\leq \infty$. The hypothesis on $\Gamma$ can be weakened considerably, and we will state our result in full in the next section. We note that Theorem \ref{t.mainintro1} and the result of Foreman, Rudolph and Weiss stand in contrast to the recent result of Hjorth and T\"ornquist \cite{hjto11}, where it is shown that conjugacy of unitary representations of any countably infinite discrete $\Gamma$ is always Borel.

\medskip

In addition to conjugacy, there are two other important equivalence relations for the p.m.p. actions that merit close consideration, namely \emph{orbit equivalence} and \emph{von Neumann equivalence} (also known as $W^*$-equivalence.) Let $(X,\mu)$ be a standard Borel probability space\footnote{In this paper all standard Borel probability spaces are non-atomic unless otherwise stated}, and let $\sigma_0, \sigma_1:\Gamma\actson (X,\mu)$ be measure preserving actions. Denote by $E_{\sigma_i}$ the orbit equivalence relation induced by $\sigma_i$, $i\in\{0,1\}$. Recall that $\sigma_0$ and $\sigma_1$ are \emph{orbit equivalent}, written $\sigma_0\oeq\sigma_1$, if there is a measure preserving Borel bijection $T:X\to X$ such that
$$
x E_{\sigma_0} x'\iff T(x) E_{\sigma_1} T(x')
$$
for almost all $x,x'\in X$. Recall also that that $\sigma_0$ and $\sigma_1$ are said to be \emph{von Neumann equivalent}, written $\sigma_0\vneq\sigma_1$, if the associated group-measure space von Neumann algebras $L^\infty(X)\rtimes_{\sigma_0}\Gamma$ and $L^\infty(X)\rtimes_{\sigma_1}\Gamma$ are isomorphic (see e.g. \cite{popa07} for a thorough discussion of von Neumann equivalence.)

When $\Gamma$ is amenable, it was shown in \cite{orwe80} and \cite{cofewe81} that all ergodic p.m.p. $\Gamma$-actions are both orbit equivalent and von Neumann equivalent. However, it has recently been shown that when $\Gamma$ is non-amenable then $\Gamma$ admits uncountably many orbit inequivalent (see \cite{epstein07}) and von Neumann inequivalent (see \cite{ioana11}) ergodic, a.e. free p.m.p. actions. The natural question whether orbit equivalence of a.e. free p.m.p. $\Gamma$-actions is Borel or analytic was raised by Kechris in \cite[18.(IVb)]{kechris10} along with Problem 1.1 above. We will prove the following:

\begin{theorem}\label{t.mainintro2}
Let $\Gamma$ be a countable discrete group containing a non-amenable free group $\F_n$, $2\leq n\leq\infty$, as a normal subgroup. Then orbit equivalence and von Neumann equivalence of weakly mixing a.e. free p.m.p. actions of $\Gamma$ are analytic relations, but they are not Borel.
\end{theorem}

Again, the assumptions on $\Gamma$ can be weakened, and we state our results in full in \S 2. Furthermore, we note that the proofs of both Theorem \ref{t.mainintro1} and \ref{t.mainintro2} use entirely different techniques than those used by Foreman, Rudolph and Weiss in \cite{foruwe11}. Namely, our proofs use rigidity techniques, and rely in particular on Popa's cocycle superrigidity theorems \cite{popa07,popa08}. Our arguments are also closer to \cite{tornquist11}, where it was proven that conjugacy and orbit equivalence are complete analytic equivalence relations (on the weakly mixing a.e. free actions) when $\Gamma$ is a countably infinite group with the relative property (T).

As a consequence of Theorem \ref{t.mainintro2}, we obtain the following:

\begin{theorem}\label{t.mainintro3}
The isomorphism relation for separably acting factors of type $\II_\infty$ and type $\III_\lambda$, for each $0\leq\lambda\leq 1$, is analytic but not Borel when the space of separably acting factors is given the Effros Borel structure.
\end{theorem}

It was shown in \cite{sato09} that the isomorphism relation for separably acting factors of type $\II_1$ is in fact complete analytic, however the argument there did not extend to factors of type $\II_\infty$ and $\III_\lambda$.

\medskip

The proof of both Theorem \ref{t.mainintro2} and \ref{t.mainintro3} relies crucially on establishing a technical result in the theory of Borel reducibility, Theorem \ref{t.ctblto1}, which shows that there is a sequence $E_\alpha$, $\alpha<\omega_1$, of Borel equivalence relations which is increasing and unbounded in the class Borel equivalence relations, when this class is ordered under the relation of \emph{countable-to-1} Borel reductions. This result, which is proved using Stern's absoluteness method from \cite{stern84}, generalizes a similar result due to Harrington for the usual Borel reducibility hierarchy.

\medskip

The paper is organized as follows: \S 2 is dedicated to preliminaries and background, as well as the statement of our results in full: We review Popa's cocycle superrigidity theorems and the related concepts, and we also briefly review the descriptive set theory (``Global theory'') of measure preserving actions, as well as the concept of Borel reducibility. In \S 3 we introduce a variant of the 1-cohomology group, called the \emph{relative 1-cohomology group}, which will be our main tool to distinguish actions up to conjugacy. In \S 4 we compute the relative 1-cohomology group for certain families of actions. \S 5 is dedicated to establishing the technical result on countable-to-1 Borel reductions described above. Finally, in \S 6 we combine the results of \S 4 and \S 5 to prove Theorem \ref{t.mainintro1}, \ref{t.mainintro2} and \ref{t.mainintro3}.

{\it Acknowledgements.} We wish to thank Alexander Kechris for useful discussions through the early stages of this work. We also thank Benjamin Weiss for his remarks on a version of this paper that was circulated in the summer of 2011.

Asger T\"ornquist is grateful for the kind hospitality and support he received during a visit to Caltech in February 2010 where part of the work for the paper was done. Thanks are also due to the Kurt G\"odel Research Center in Vienna, where part of the paper was written while Asger T\"ornquist was employed there under FWF project P 19375-N18.

Asger T\"ornquist was supported by a Sapere Aude fellowship (level 2) from Denmark's Natural Sciences Research Council, no. 10-082689/FNU, and a Marie Curie re-integration grant, no. IRG-249167, from the European Union.

\section{Preliminaries and statement of results}

\subsection{Global theory} The main reference for the global theory of measure preserving actions is Kechris' book \cite{kechris10}.

Let $(X,\mu)$ be a standard Borel probability space. The group of measure preserving transformations of $(X,\mu)$ is denoted $\Aut(X,\mu)$. We equip this group with the weak topology, i.e., the initial topology making all the maps
$$
\Aut(X,\mu)\to [0,1]: T\mapsto \mu(T(A)\triangle B)
$$
continuous, where $A,B\subseteq X$ are Borel sets. This makes $\Aut(X,\mu)$ a Polish group. Let $\Gamma$ be a countable group. The set
$$
\act(\Gamma,X,\mu)=\{\sigma\in\Aut(X,\mu)^\Gamma: (\forall\gamma_1,\gamma_2\in\Gamma) \sigma(\gamma_1\gamma_2)=\sigma(\gamma_1)\sigma(\gamma_2)\}
$$
is closed in the product topology, and each $\sigma\in\act(\Gamma,X,\mu)$ defines a measure preserving action of $\Gamma$ on $(X,\mu)$ almost everywhere. We call the Polish space $A(\Gamma,X,\mu)$ the 
\emph{space of measure preserving actions of $\Gamma$}.\footnote{More correctly, the space $\act(\Gamma,X,\mu)$ should be called the space of boolean actions of $\Gamma$, since each $\sigma\in A(\Gamma,X,\mu)$ only determines an action almost everywhere.} Following \cite{kechris10}, let FR$(\Gamma,X,\mu)$ denote the subset of $\act(\Gamma,X,\mu)$ consisting of a.e. free $\Gamma$-actions, WMIX$(\Gamma,X,\mu)$ the set of weakly mixing $\Gamma$-actions, and let
$$
\actfwm(\Gamma,X,\mu)={\rm FR}(X,\Gamma,\mu)\cap {\rm WMIX}(\Gamma,X,\mu)
$$
be the set of a.e. free weakly mixing actions. These sets are Borel. The conjugacy relation, denoted $\conj$, in $\act(\Gamma,X,\mu)$ is defined by
$$
\sigma_0\conj\sigma_1\iff (\exists T\in\Aut(X,\mu))(\forall\gamma\in\Gamma) T\sigma_0(\gamma)T^{-1}=\sigma_1(\gamma).
$$
The relations $\oeq$ and $\vneq$ were defined in the introduction.

\begin{lemma}
The relations $\conj$, $\oeq$ and $\vneq$ are analytic.
\end{lemma}

\begin{proof}
For $\conj$ this is immediate from the definition. For $\oeq$ this was proven in \cite[Lemma 5.5]{tornquist11}. For $\vneq$, it follows by \cite[Lemma 5]{sato09} and the proof of \cite[Corollary 15]{sato09}.
\end{proof}

It should be emphasized that the previous Lemma does \emph{nothing} to rule out the possibility that $\conj$, $\oeq$ or $\vneq$ could be Borel.

\medskip

\noindent{\it Notation}: For any set $X$, we use the notation $\sigma:\Gamma\actson X$ to mean that $\sigma$ is an action of $\Gamma$ on $X$. If $(X,\mu)$ is a standard Borel probability space, we write $\sigma:\Gamma\actson (X,\mu)$  to mean that $\sigma$ is a Borel action of $\Gamma$ on $X$ which preserves $\mu$ (thus this is almost, but not quite, synonymous with the statement $\sigma\in\act(\Gamma,X,\mu)$.) If $\Lambda\leq\Gamma$ and $\sigma:\Gamma\actson X$ is an action, then $\sigma\restrict\Lambda$ and $\Lambda\actson^\sigma X$ are both notation for the \emph{restriction} of the action $\sigma$ to the subgroup $\Lambda$. For $\sigma:\Gamma \actson X$, we will write $\sigma(\gamma)(x)$ or $\gamma\cdot_{\sigma} x$ for the action by $\sigma$ of $\gamma$ on $x\in X$.

\subsection{Cocycle superrigidity for malleable actions} We now review the cocycle superrigidity Theorems of Popa (including the relative formulations given by Furman in \cite{furman07}, which may be consulted for further background), and fix our terminology. At the end of this section we state the main results of the paper in full.

Let $\Gamma$ be a countable discrete group, $(X,\mu)$ a standard Borel probability space, $\sigma:\Gamma\actson (X,\mu)$ a p.m.p. action, and let $G$ be a topological group. Recall that a cocycle of $\sigma$ with target group $G$ is a measurable map $\alpha:\Gamma\times X\to G$ that satisfies the cocycle identity
$$
(\forall\gamma_1,\gamma_2\in\Gamma)\alpha(\gamma_1\gamma_2,x)=\alpha(\gamma_1,\gamma_2\cdot_\sigma x)\alpha(\gamma_2,x)
$$
almost everywhere.

Let $\sigma:\Gamma\actson (X,\mu)$ as before, let $(Y,\nu)$ be a standard probability space, and suppose $\rho:\Gamma\actson Y$ is a quotient (extension) of $\sigma$ with quotient map $p:X\to Y$. Let $Y\mapsto P(X):y\mapsto\mu_y$ be the disintegration of $\mu$ with respect to $p$ (see e.g. \cite[17.35]{kechris95}.) The diagonal product action $\sigma\times\sigma:\Gamma\actson X\times X$ preserves the \emph{fibered product space}\footnote{Our notation differs slightly from that of Furman, who uses $X\times_Y X$ and $\mu\times_\nu\mu$ for what we have chosen to call $X\times_p X$ and $\mu\times_p\mu$. This makes explicit the dependence on the quotient map $p$.}
$$
X\times_p X=\{(x_1,x_2)\in X\times X: p(x_1)=p(x_2)\}
$$
and preserves the measure
$$
\mu\times_p\mu=\int \mu_y\times\mu_y d\nu(y),
$$
see \cite{furman07} or \cite[6.3]{glasner03}. We denote this action by $\sigma\times_p\sigma:\Gamma\actson X\times_p X$.

\begin{definition}[See {\cite[1.a]{furman07}}] Let $\sigma:\Gamma\actson (X,\mu)$, $\rho:\Gamma\actson (Y,\nu)$, $p:X\to Y$ be as above, and let $\mathscr C$ be a class of topological groups (``target groups'').

%\begin{enumerate}[1.] 

1. We say that $\sigma$ is \emph{weakly mixing relative to} $(p,\rho)$ if $\sigma\times_p\sigma\actson (X\times_p X,\mu\times_p\mu)$ is ergodic. Note that $\sigma$ is weakly mixing relative to the trivial quotient (i.e., where $\rho$ is the action on a single point) precisely when it is weakly mixing.

2. We say that $\sigma$ is \emph{malleable relative to} $(p,\rho)$ if the flip $(x_1,x_2)\mapsto (x_2,x_1)$ is in the connected component of the identity in $\Aut(X\times_p X,\mu\times_p\mu)$. We will say that $\sigma$ is \emph{malleable} if it is malleable relative to the trivial quotient.

3. The action $\sigma$ is said to be \emph{$\mathscr C$-cocycle superrigid relative to $(p,\rho)$} if any measurable cocycle $\alpha:\Gamma\times X\to G$ with target group $G\in\mathscr C$ is cohomologous to a cocycle of $\rho$. That is, there is $f:X\to G$ measurable and a measurable cocycle $\vartheta: \Gamma\times Y\to G$ of the action $\rho$ such that
$$
\vartheta(\gamma,p(x))=f(\gamma\cdot_\sigma x)\alpha(\gamma,x)f(x)^{-1}.
$$

4. A subgroup $\Lambda\leq\Gamma$ is said to be \emph{weakly normal} (or \emph{w-normal}) if there is a wellfounded chain of subgroups $\{\Lambda_i:i\in I\}$ ordered by inclusion, with $\Lambda$ the least element and $\Gamma$ the largest, and so that $\bigcup \{\Lambda_i:\Lambda_i\subset\Lambda_j, i\in I\}$ is normal in $\Lambda_j$, for all $j\in I$.
%\end{enumerate}
\end{definition}

The main classes of target groups we consider are $\{\T\}$ (where $\T=\{z\in\C:|z|=1\}$), the class of countable groups $\mathscr G_{\text{ctbl}}$, and the class $\mathscr U_{\text{fin}}$ of groups realizable as a closed subgroup of the unitary group of a finite separable von Neumann algebra. The following theorem is Popa's cocycle superrigidity theorem for w-rigid groups (with Furman's relative formulation from \cite{furman07}, which we make crucial use of in this paper.)

\begin{theorem}[Popa, {\cite{popa07}}]\label{popa_csr}
Let $\mathscr C=\mathscr U_{\rm fin}$. Let $\Gamma$ be a countable discrete group, and let $\sigma:\Gamma\actson (X,\mu)$ be a measure preserving action with quotient $\rho:\Gamma\actson (Y,\nu)$ and quotient map $p:X\to Y$. Suppose that there is an w-normal subgroup $\Lambda\leq\Gamma$ such that $\Gamma$ has the relative property (T) over $\Lambda$. If $\sigma\restrict\Lambda$ is weakly mixing and malleable relative to $(p,\rho)$ then $\sigma$ is $\mathscr C$-cocycle superrigid relative to $(p,\rho)$. In particular, if $\sigma$ is malleable and weakly mixing on $\Lambda$ then any measurable cocycle with target group $G\in\mathscr U_{\text{fin}}$ is cohomologous to a group homomorphism $\vartheta:\Gamma\to G$.
\end{theorem}

\begin{remark}

In \cite{popa08}, Popa proved a second cocycle superrigidity theorem under a rather different set of hypotheses. While Theorem \ref{popa_csr} suffices for most of our purposes, the second cocycle superrigidity theorem does allow us to carry out our calculations in \S 4 under a different set of hypothesis. We refer the reader to Popa's paper \cite{popa08} for the statement of the second cocycle superrigidity theorem, as well as the definition of the notions \emph{s-malleable} actions and \emph{spectral gap}.

Recently, Peterson and Sinclair \cite{pesi11} have obtained a cocycle superrigidity theorem which can be viewed as a simultaneous generalization of both of Popa's cocycle superrigidity theorems.
\end{remark}

\begin{definition}
Let $\mathscr C$ be a class of target groups. We will say that a group $\Gamma$ has \emph{strongly} $\mathscr C$-superrigid malleable (or s-malleable) weakly mixing actions if the conclusion of Theorem \ref{popa_csr} holds for any $\sigma:\Gamma\actson (X,\mu)$ and any quotient $\rho:\Gamma\actson (Y,\nu)$ with respect to which $\sigma$ is weakly mixing and malleable.
\end{definition}

With this definition, it follows from Theorem \ref{popa_csr} that any group $\Gamma$ with property (T) has strongly $\mathscr U_{\rm fin}$-superrigid malleable weakly mixing actions. It would of course be possible to introduce a relative version of the previous definition which captures Theorem \ref{popa_csr} exactly, but this would make our terminology unnecessarily cumbersome below, so we will refrain from doing so. 

\begin{definition}
Let $\Lambda$ be a countably infinite discrete group and suppose $\Delta\leq\Lambda$. We will say that an action $\sigma_0:\Lambda\actson \N$ is \emph{$\Delta$-suitable} if all $\sigma_0$-orbits are infinite, but $\sigma_0\restrict\Delta$ has a finite orbit.
\end{definition}

A typical example of a suitable action will be the following: Suppose $\Delta\lhd\Lambda$ as above, and that $\Lambda\leq\Gamma$, where $\Gamma$ is some larger countably infinite discrete group. Then if $\Lambda\lhd\Gamma$, then the action of $\Lambda$ of the left cosets $\Gamma/\Delta$ is easily seen to be a $\Delta$-suitable action of $\Lambda$. (See Corollary \ref{c.almostnormal} for a slightly more general statement.)

Having introduced these concepts, we can now state the results of this paper in full.

\begin{theorem}\label{t.mainthm1}
Let $\Gamma$ be a countably infinite discrete group, and suppose that $\Delta\lhd\Lambda\leq\Gamma$ are subgroups such that $\Lambda/\Delta$ has strongly $\{\T\}$-cocycle superrigid weakly mixing malleable actions. Suppose $\Gamma$ admits an action $\sigma_0:\Gamma\actson\N$ such that $\sigma_0\restrict\Lambda$ is $\Delta$-suitable. Then the conjugacy relation for measure preserving weakly mixing a.e. free $\Gamma$-actions is complete analytic (in the sense of \cite[22.9]{kechris95}) as a subset of $\actfwm(\Gamma,X,\mu)^2$.
\end{theorem}

\begin{theorem}\label{t.mainthm2}
Let $\Gamma$ be a countably infinite discrete group, and suppose that $\Gamma$ contains a subgroup $\Lambda$ isomorphic to $\F_n$ for some $2\leq n\leq\infty$. Suppose that there is $\Delta\lhd\Lambda$ such that $\Lambda/\Delta$ has strongly $\{\T\}$-cocycle superrigid weakly mixing malleable actions and that $\Gamma$ admits an action $\sigma_0:\Gamma\actson\N$ such that $\sigma_0\restrict\Lambda$ is $\Delta$-suitable. Then the relations of orbit equivalence and von Neumann equivalence of measure preserving a.e. free weakly mixing actions of $\Gamma$ are analytic, but neither are Borel.
\end{theorem}

\subsection{Borel reducibility.}

The proofs of Theorems \ref{t.mainthm1} and \ref{t.mainthm2} use techniques from the theory of \emph{Borel reducibility} in their proof, and we now review the basics of this subject that are most important for this paper. This also allows us to restate Theorem \ref{t.mainthm1} and \ref{t.mainthm2} in their final form.

\begin{definition}
Let $X$ and $Y$ be Polish spaces, and let $E$ and $F$ be equivalence relations on $X$ and $Y$, respectively. 

1. We say that $E$ is \emph{Borel reducible} to $F$, written $E\leq_B F$, if there is a Borel function $f:X\to Y$ such that
$$
(\forall x,x'\in X) xEx'\iff f(x) F f(x').
$$
Note that $f$ induces a 1-1 map from the quotient space $X/E$ to $Y/F$.

2. We say that $E$ is \emph{Borel countable-to-1 reducible} to $F$, written $E\leq_B^\N F$, if there is a Borel function function $f:X\to Y$ such that
$$
(\forall x,x'\in X) xEx'\implies f(x) F f(x'),
$$
and for all $y\in Y$ the inverse image $f^{-1}([y]_F)$ of an $F$-equivalence class consists of at most countably many $E$-classes. Note that $f$ induces a countable-to-1 map from $X/E$ to $Y/F$.
\end{definition}

The isomorphism relation for countable abelian groups, and in particular for torsion free countable abelian groups, plays a key role in this paper. We let $\bf ABEL$ and $\bf TFA$ denote the Polish spaces of abelian and torsion free abelian groups with underlying set $\N$. That is, $\bf ABEL$ consists of the triples 
$(c,i,e)\in \N^{\N\times\N}\times\N^\N\times\N$ such that the operation $n+_c m=c(n,m)$ defines an Abelian group structure on $\N$ with inverse given by $i(n)$, and $e$ the neutral element. Then ${\bf ABEL}$ is a closed subset of $\N^{\N\times\N}\times\N^\N\times\N$ when this space has the product topology, and so ${\bf ABEL}$ is Polish in this topology. The set ${\bf TFA}\subset{\bf ABEL}$ is similarly seen to form a Polish space. The isomorphism relation in ${\bf ABEL}$ and ${\bf TFA}$ is denoted $\simeq^{\bf ABEL}$ and $\simeq^{\bf TFA}$, respectively.

\medskip

In \S 6 we will prove the following two theorems, which we will see imply Theorem \ref{t.mainthm1} and \ref{t.mainthm2}. Below and elsewhere we use the following notational convention: For an equivalence relation $E$ on a set $X$ and $A\subseteq X$, we write $E^A$ for the restriction of $E$ to $A$.

\begin{theorem}\label{t.mainthm1v2}
Under the hypotheses of Theorem \ref{t.mainthm1} we have 
$$
\simeq^{\bf TFA}\leq_B\conj^{\actfwm(\Gamma,X,\mu)}.
$$
\end{theorem}

\begin{theorem}\label{t.mainthm2v2}
Under the hypothesis of Theorem \ref{t.mainthm2} we have 
$$\simeq^{\bf TFA}\leq_B^{\N}\oeq^{\actfwm(\Gamma,X,\mu)}
$$ 
and
$$
\simeq^{\bf TFA}\leq_B^{\N}\vneq^{\actfwm(\Gamma,X,\mu)}.
$$
\end{theorem}

% SUBSECTION: CONJUGACY AND PROBLEM 1

\subsection{The conjugacy problem and Problem \ref{pr.kechris}} Before moving on to the proofs, we briefly discuss the heuristics surrounding the relationship between Problem \ref{pr.kechris} and the somewhat vaguely stated conjugacy problem. We also refer to the introduction of \cite{foruwe11}, though the discussion there differs from ours in some respects. We discuss the situation only for conjugacy where we are aided by having a natural group action that induces the equivalence relation, but it applies with only a small modification to orbit equivalence and von Neumann equivalence as well using the fact that these equivalence relations reduce to orbit equivalence relations induced by the unitary group of $\ell^2$, as was proved in \cite{sato09}.

The conjugacy problem asks for a \emph{method} for deciding a \emph{yes-or-no} question, and in this sense it is analogous to the classical decision problems in recursion theory. The question is how good a decision procedure for the conjugacy problem (or the analogous problem for deciding orbit equivalence and von Neumann equivalence) can be, and this is why, as we explain below, Problem \ref{pr.kechris} becomes relevant to the conjugacy problem.

We first need to introduce some standard notation related to trees in descriptive set theory. This will also be useful later in \S 5.

Let $\N^{<\N}$ denote the set of all finite sequences in $\N$, i.e. functions $s:k\to\N$ for some $k\in\N_0$ (where make the identification $k=\{0,\ldots, k-1\}$. Note that we have included the empty sequence, denoted $\emptyset$.) We will write $\length(s)=k$ if $\dom(s)=k$, and for $j\in\N$ we denote by $s\restrict j$ the restriction of $s$ to $j\cap k$. For $s,t\in\N^{<\N}$, we write $t\subseteq s$ if $s$ extends $t$, in which case we also say that $t$ is an \emph{initial segment} of $s$, and we write $s^\frown t$ for the sequence obtained by appending $t$ at the end of $s$. For convenience, we identify $i\in\N$ with the sequence $\langle i\rangle$ of length 1, and so $s^\frown i$ is the sequence $s$ with $i$ appended at the end. A \emph{tree} on $\N$ is a set $T\subseteq\N^{<\N}$ which is closed under initial segments. The elements of $T$ are called \emph{nodes}; an element in $T$ which has no extension in $T$ is called a \emph{terminal node}. An \emph{infinite branch} through a tree $T$ is an element $x\in\N^{\N_0}$ such that $x\restrict k\in T$ for all $k\in\N_0$. A tree on $\N$ is \emph{well-founded} if it has no infinite branches; otherwise it is \emph{ill-founded}. We let ${\rm Tree}(\N)$ denote the set of all trees on $\N$. (\cite[Chapter 2]{kechris95} is a good elementary reference on the basics of trees in descriptive set theory.)

Fix a countable discrete group $\Gamma$ and a standard Borel probability space $(X,\mu)$. The conjugacy relation $\conj$ on $\act(\Gamma,X,\mu)$ is induced by the continuous action
$$
(T\cdot\sigma)(\gamma)=T\sigma(\gamma)T^{-1}.
$$
We will now device a decision procedure for the yes-or-no problem of checking conjugacy only using this fact. For this, fix a complete compatible metric $\delta$ on $\act(\Gamma,X,\mu)$ bounded by 1.

Fix $\sigma_0,\sigma_1\in\act(\Gamma,X,\mu)$. Let $S_n$ enumerate a dense set in $\Aut(X,\mu)$ and let $d$ be a complete compatible metric on $\Aut(X,\mu)$ which is bounded by 1. We define a sequence $T_n^{\sigma_0,\sigma_1}\subset\N^n$ for $n\in\N_0$ by recursion on $n$ as follows: $T_0^{\sigma_0,\sigma_1}=\{\emptyset\}$, $T_1^{\sigma_0,\sigma_1}=\N^1$, and if $T_n^{\sigma_0,\sigma_1}$ has been defined, then we define for $s\in\N^{n+1}$
$$
s\in T_{n+1}^{\sigma_0,\sigma_1}\iff s\restrict n\in T_n^{\sigma_0,\sigma_1}\wedge \delta(S_{s(n)}\cdot \sigma_0, \sigma_1)\leq \frac 1 {2^n}\wedge d(S_{s(n)},S_{s(n-1)})\leq \frac 1 {2^n}.
$$
It is then clear that
$$
T^{\sigma_0,\sigma_1}=\bigcup_{n=0}^\infty T_n^{\sigma_0,\sigma_1}
$$
is a tree, which we call the \emph{search tree}: The branches of $T^{\sigma_0,\sigma_1}$ index attempts at finding a Cauchy sequence in $\Aut(X,\mu)$ which converges to a transformation conjugating $\sigma_0$ and $\sigma_1$.

Observe then that the search tree $T^{\sigma_0,\sigma_1}$ has an infinite branch if and only if $\sigma_0$ and $\sigma_1$ are conjugate: Namely, if $x\in\N^\N$ is an infinite branch then $S_{x(n)}$ is a Cauchy sequence in $\Aut(X,\mu)$, and $S_{x(n)}\cdot \sigma_0\to \sigma_1$ as $n\to\infty$. Conversely, if $S\cdot\sigma_0=\sigma_1$ for some $S\in\Aut(X,\mu)$ then clearly there is some $x\in\N^{\N_0}$ such that $\delta(S_{x(n)}\cdot\sigma_0,\sigma_1)\leq \frac 1 {2^n}$ for all $n\in\N$, and $d(S_{x(n)},S)\leq\frac 1 {2^{n+1}}$ for all $n\in\N_0$, from which it easily follows that $x\in T^{\sigma_0,\sigma_1}$.

In this way we have reduced the conjugacy problem to determining whether or not there is an infinite branch through a tree on $\N$. It is of course well-known that membership in \emph{any} analytic set can be tested in this way: The set of all ill-founded trees on $\N$ is a \emph{complete analytic} set, in the sense that if $A$ is any analytic set in a Polish space $X$, then there is a Borel function $f:X\to{\rm Tree}(\N)$ such that $x\in A$ iff $f(x)$ is ill-founded. So the method above tells us nothing special about the conjugacy problem, and we can regard it as the \emph{worst possible} method for deciding conjugacy. The question therefore is if there is a \emph{better} method than the method described above.

If the conjugacy relation were a Borel relation, then the answer to that question would be `yes'. Namely, we could fix a description of the Borel set of pairs $(\sigma_0,\sigma_1)$ which are conjugate (that is, a description of how it is build up from basic open sets), such as for instance the description using ``Borel codes'' found in \S 5 of this paper. The description of how the Borel set is build from basic open set can be thought of as a recipe for how, in countably many steps, the problem of determining if the pair $(\sigma_0,\sigma_1)$ are conjugate to determining membership in countably many basic open sets (which we assume is an easy task.)

The important difference between the method given by a Borel description, and the general method of search trees, is that the Borel method requires only countable resources to be at our disposal, whereas the search tree method potentially requires access to uncountable resources, since the well-founded trees on $\N$ can define \emph{all} countable ordinals, see e.g. \cite[2.G]{kechris95}. If the conjugacy problem is analytic but not Borel, then no general method for determining conjugacy that use only countable resources exists, since any definition of the search tree must produce trees defining arbitrarily large countable ordinals (otherwise it \emph{would} be Borel.) And if it is \emph{complete analytic}, then the algorithm described above, which means checking if the search tree is ill-founded or not, is \emph{best possible}, in the sense that deciding if a tree on $\N$ has an infinite branch can be reduced to it.\footnote{The difference between analytic non-Borel sets and complete analytic sets is not as big as it may appear at a first glance: Under relatively modest and reasonable extra set-theoretic hypotheses it disappears entirely, see Remark \ref{r.sdf}.}

% SECTION

\section{The relative 1-cohomology group of a measure preserving action}

Let $\Lambda$ be a countable group, and let $\sigma:\Lambda\actson (X,\mu)$ be a probability measure preserving action. A 1-cocycle is a cocycle with target group $\T=\{z\in\C:|z|=1\}$, i.e., a measurable map $\alpha:\Lambda\times X\to\T$ such that for all $\gamma_0,\gamma_1\in\Lambda$ and almost all $x\in X$ we have
$$
\alpha(\gamma_0\gamma_1,x)=\alpha(\gamma_0,\gamma_1\cdot_\sigma x)\alpha(\gamma_1,x).
$$
The set of measurable 1-cocycles is denoted $Z^1(\sigma)$ and forms a subgroup under pointwise multiplication. We give $Z^1(\sigma)$ the topology it inherits from $L^1(\Lambda\times X)$, which makes it a Polish group. A 1-coboundary is a 1-cocycle of the form
$$
\alpha(\gamma,x)=f(\gamma\cdot_\sigma x)f(x)^*
$$
for some measurable $f:X\to\T$, and the set of 1-coboundaries is a (not necessarily closed) subgroup of $Z^1(\sigma)$, denoted $B^1(\sigma)$. The \emph{1-cohomology group} of $\sigma$ is defined as $H^1(\sigma)=Z^1(\sigma)/B^1(\sigma)$.

Suppose now $\Delta<\Lambda$ is a subgroup. Let $\sigma\restrict\Delta$ denote the restriction of $\sigma$ to $\Delta$. Then we have the restriction map
$$
\varrho:Z^1(\sigma)\to Z^1(\sigma\restrict\Delta):\alpha\mapsto\alpha\restrict \Delta\times X.
$$

\begin{definition}
Let $\Delta<\Lambda$ be countable groups and $\sigma:\Lambda\actson X$ as above. We define
$$
Z^{1}_{:\Delta}(\sigma)=\{\alpha\in Z^1(\sigma): \alpha\restrict\Delta\times X=1\}=\ker(\varrho)
$$
and $B^1_{:\Delta}(\sigma)=B^1(\sigma)\cap Z^1_{:\Delta}(\sigma)$. Elements of $Z^1_{:\Delta}(\sigma)$ will be called \emph{$\Delta$-trivial} 1-cocycles, and elements of $B^1_{:\Delta}(\sigma)$ will be called \emph{$\Delta$-trivial} 1-coboundaries. The $\Delta$-\emph{relative} 1-cohomology group of $\sigma$ is then defined as $H^1_{:\Delta}(\sigma)=Z^1_{:\Delta}(\sigma)/B^1_{:\Delta}(\sigma)$.
\end{definition}

From this definition little else is clear except that $H^1_{:\Delta}(\sigma)$ is a conjugacy invariant of the action $\sigma$. Further, we have:

% LEMMA

\begin{lemma}\label{l.ergtrivial}
If $\Delta$ is a normal subgroup of $\Lambda$ and $\sigma\restrict\Delta$ is ergodic, then $B^1_{:\Delta}(\sigma)=\{1\}$ and 
$$
Z^1_{:\Delta}(\sigma)\simeq H^1_{:\Delta}(\sigma)\simeq\Char(\Lambda/\Delta).
$$
\end{lemma}

\begin{proof}
Let $\alpha\in Z^1_{:\Delta}(\sigma)$. For $\delta,\delta'\in\Delta$ and $\gamma\in\Lambda$ such that $\gamma\delta=\delta'\gamma$ we have
$$
\alpha(\gamma,\delta\cdot_\sigma x)=\alpha(\gamma\delta,x)\alpha(\delta,x)^{-1}=\alpha(\delta'\gamma,x)=\alpha(\delta',\gamma\cdot_\sigma x)\alpha(\gamma,x)=\alpha(\gamma,x)
$$
showing that for each $\gamma\in\Lambda$ the function $x\mapsto\alpha(\gamma,x)$ is $\Delta$-invariant, therefore is constant a.e. For a $\Delta$-trivial 1-coboundary this invariance amounts to
$$
f(\gamma\cdot x)=c_\gamma f(x) \ \text{ (a.e.)}
$$
for some $c_\gamma\in\C$. For $\delta\in\Delta$ we must have $c_\delta=1$, and so $f(\delta\cdot x)=f(x)$, whence by the ergodicity of $\sigma\restrict \Delta$ we must have $f=1$.
\end{proof}

The previous Lemma indicates that $H^1_{:\Delta}(\sigma)$ is only potentially interesting as an invariant when $\sigma\restrict\Delta$ is \emph{not} ergodic. As we will see, $H^1_{:\Delta}(\sigma)$ is an invariant of the pattern of \emph{non-ergodicity} of $\sigma\restrict\Delta$. It turns out that $H^1_{:\Delta}(\sigma)$ can be controlled in some constructions, making it a useful conjugacy invariant.

\subsection{The canonical action of $\Lambda$ on the ergodic components $\sigma\restrict\Delta$.} 
For the purpose of this section, we fix a standard Borel space $X$, a countable discrete group $\Lambda$ with a normal subgroup $\Delta\lhd\Lambda$, $\kappa:\Lambda\to\Lambda/\Delta$ the canonical epimorphism, and $\sigma:\Lambda\actson X$ a Borel action on $X$. Let $P(X)$ denote the Polish space of Borel probability measures on $X$. Following the notation of \cite{mike04}, we denote by $\mathcal I_{\sigma}(X)\subseteq P(X)$ the set of $\sigma$-invariant measures on $X$, and by $\mathcal E_\sigma(X)\subseteq\mathcal I_\sigma(X)$ the set of $\sigma$-invariant ergodic measures on $X$. These can be seen to form Borel subsets of $P(X)$.

The group $\Lambda$ acts on $P(X)$ by defining $(\gamma\cdot\mu)(A)=\mu(\gamma^{-1}\cdot_\sigma A)$ for all Borel $A\subseteq X$. When $\Delta\lhd\Lambda$ is a normal subgroup, then $\mathcal I_{\sigma\vrestrict\Delta}(X)$ and $\mathcal E_{\sigma\vrestrict\Delta}(X)$ are clearly invariant under the action of $\Lambda$ on $P(X)$.

Suppose now $\mu\in\mathcal I_\sigma(X)$ is a fixed $\sigma$-invariant measure. Let $\pi:(X,\mu)\to (Y,\nu)$ and $Y\to \mathcal E_{\sigma\vrestrict\Delta}:y\mapsto\mu_y$ be an ergodic decomposition of $\sigma\restrict\Delta$, in the sense that $\pi:X\to Y$ is a Borel map onto $Y$ which is $\sigma\restrict\Delta$ invariant, $y\mapsto\mu_y$ is Borel and assigns to each $y\in Y$ the unique $\sigma\restrict\Delta$-invariant and ergodic measure supported on $\pi^{-1}(y)$, and the disintegration identity
$$
\mu=\int \mu_y d\nu(y)
$$
is satisfied (see \cite[Theorem 3.3]{mike04}.) Recall that this decomposition is essentially unique: If $\hat\pi:X\to Z$, $Z\to\mathcal E_{\sigma\vrestrict\Delta}: z\mapsto\mu_z$ was another such decomposition, then there is a Borel bijection $\theta:Y\to Z$ such that $\hat\pi=\theta\circ\pi$ and $\mu_y=\mu_{\theta(y)}$ for almost all $y\in Y$.

% LEMMA: QUOTIENT

\begin{lemma}\label{l.quotientaction}
With notation as above, there is a $\nu$-preserving Borel action $\sigma_\Delta:\Lambda\actson Y$ which is a factor of $\sigma$ with $\pi:X\to Y$ as factor map, i.e., 
$$
\pi(\gamma\cdot_\sigma x)=\sigma_{\Delta}(\gamma)(\pi(x)) \text{\ \ (a.e.)}
$$
The action $\sigma_{\Delta}$ is $\nu$-ergodic if and only if $\sigma$ is $\mu$-ergodic. Moreover, $\sigma_\Delta\restrict\Delta$ is the trivial action on $Y$, and so $\sigma_\Delta$ factors to an action $\bar\sigma_\Delta:\Lambda/\Delta\actson Y$ through $\kappa$.
\end{lemma}

\begin{proof}
Define $\sigma_\Delta:\Lambda\actson Y$ by
$$
\sigma_\Delta(\gamma)(y)=y'\iff \gamma\cdot\mu_y=\mu_{y'}.
$$
To see that this makes sense and defines an action (at least a.e.), note that for $\gamma\in\Lambda$ we have that $\gamma\cdot\mu_y$ is a $\Delta$-invariant ergodic measure on $\gamma\cdot\pi^{-1}(y)$. Thus $\hat\pi(x)=\pi(\gamma\cdot_\sigma x)$ and $y\mapsto\gamma\cdot\mu_y$ provides an ergodic decomposition of $\mu$, and by the uniqueness of the ergodic decomposition we have $\gamma\cdot\mu_{\pi(x)}=\mu_{\hat\pi(x)}$ for $\mu$-almost all $x\in X$. Thus $\sigma_\Delta$ is defined a.e. and satisfies $\sigma_\Delta(\gamma)(\pi(x))=\pi(\gamma\cdot_\sigma x)$. Finally, for any measurable $A\subseteq Y$ we have 
$$
\nu(\sigma_\Delta(\gamma)(A))=\mu(\pi^{-1}(\sigma_\Delta(\gamma)(A)))=\mu(\sigma(\gamma)(\pi^{-1}(A)))=\nu(A),
$$
so that $\sigma_\Delta$ is $\nu$-preserving. The remaining claims now follow easily.
\end{proof}

\begin{definition}\label{d.quotientaction}
The action defined in the previous Lemma will be called the \emph{canonical action of $\Lambda$ }(respectively $\Lambda/\Delta$) \emph{on the ergodic components of $\sigma\restrict\Delta$}, and it will always be denoted $\sigma_\Delta$ (respectively $\bar\sigma_\Delta$.) The space $Y$ will be denoted $X_\Delta$ and the factor map $\pi$ will in general be denoted by $p_{X_\Delta}:X\to X_\Delta$ later in this paper.
\end{definition}

\subsection{The groups $H^1_{:\Delta}(\sigma)$ and $H^1(\bar\sigma_\Delta)$.}

For ease of notation, let $\bar\gamma=\kappa(\gamma)$. The factor map $p_{X_\Delta}=\pi:X\to Y$ provides a natural homomorphism $\tilde\pi:Z^1(\bar\sigma_\Delta)\to Z^1_{:\Delta}(\sigma)$ by
$$
\tilde\pi(\alpha)(\gamma,x)=\alpha(\bar\gamma,\pi(x)).
$$
We call $\tilde\pi$ the \emph{canonical} homomorphism in this context. Note that $\tilde\pi$ is continuous. The next lemma explains the significance of $H^1_{:\Delta}$ in terms of the action $\bar\sigma_\Delta$.

% LEMMA: Canonical isomorphism of H^1_{:\Delta} and H^1(\bar\sigma_\Delta

\begin{lemma}\label{l.canisom}
The canonical homomorphism $\tilde\pi: Z^1(\bar\sigma_\Delta)\to Z^1_{:\Delta}(\sigma)$ is an isomorphism which satisfies $\tilde\pi(B^1(\bar\sigma_\Delta))=B^1_{:\Delta}(\sigma)$. Thus $\tilde\pi$ factors to an isomorphism $\hat\pi: H^1(\bar\sigma_\Delta)\to H^1_{:\Delta}(\sigma)$.
\end{lemma}

\begin{proof}
If $\alpha\in Z^1(\bar\sigma_\Delta)$ and $\tilde\pi(\alpha)=1$ then clearly $\alpha=1$, so $\tilde\pi$ is injective. On the other hand, if $\beta\in Z^1_{:\Delta}(\sigma)$ then for $\gamma\in\Lambda$ and $\delta\in\Delta$ we have
$$
\beta(\gamma,\delta\cdot x)=\beta(\gamma\delta,x)\beta(\delta,x)^*=\beta(\gamma\delta\gamma^{-1}\gamma,x)=\beta(\gamma\delta\gamma^{-1},\gamma\cdot_\sigma x)\beta(\gamma,x)=\beta(\gamma,x)
$$
so that for each $\gamma\in\Lambda$ the map $x\mapsto\beta(\gamma,x)$ is $\Delta$-invariant, and therefore constant on almost every ergodic component. Since clearly $\beta(\gamma\delta,x)=\beta(\gamma,x)$, we have that $\bar\beta(\bar\gamma,\pi(x))=\beta(\gamma,x)$ defines an element of $Z^1_{:\Delta}(\sigma)$ such that $\tilde\pi(\bar\beta)=\beta$.

The inclusion $\tilde\pi(B^1(\bar\sigma_\Delta))\subseteq B^1_{:\Delta}(\sigma)$ is clear. For the other inclusion, suppose $\tilde\pi(\alpha)\in B^1_{:\Delta}(\sigma)$ and let $f:X\to\T$ be such that
$$
\tilde\pi(\alpha)(\gamma,x)=\alpha(\bar\gamma,\pi(x))=f(\gamma\cdot_\sigma x)f(x)^*.
$$
Then $f(\delta\cdot_\sigma x)f(x)^*=1$ for $\delta\in\Delta$ and so $f$ is $\Delta$-invariant, and therefore invariant on almost every ergodic component. Thus $\alpha\in B^1(\bar\sigma_\Delta)$.
\end{proof}

We will refer to the map $\hat\pi:H^1(\bar\sigma_\Delta)\to H^1_{:\Delta}(\sigma)$ defined in the previous Lemma as the \emph{canonical} isomorphism. Since Lemma \ref{l.canisom} provides that description of $H^1_{:\Delta}(\sigma)$ in terms of the usual 1-cohomology group of the action $\bar\sigma_\Delta$ we obtain the following from \cite{schmidt80} and \cite{popa06}.
\begin{corollary}
(1) The group $B^1_{:\Delta}(\sigma)$ is closed if and only if the action $\bar\sigma_\Delta$ is strongly ergodic.

(2) If $\Lambda/\Delta$ has property (T) then $B^1_{:\Delta}(\sigma)$ is clopen and $H^1_{:\Delta}(\sigma)$ is countable and discrete.

\end{corollary}

% SECTION: CALCULATING H^1_{:\Delta}

\section{Families of actions with $H^1_{:\Delta}$ non-trivial and calculable}

In this section, we will compute the $\Delta$-relative 1-cohomology group of certain p.m.p. actions of the form $\sigma\times\rho$ under reasonably general conditions. We make our computations in a somewhat more general setting than what is narrowly needed for our applications in \S 6, where it turns out that we always have that $\rho\restrict\Delta$ is weakly mixing, in which case it follows from Lemma \ref{l.canisom} that $H^1_{:\Delta}(\sigma\times\rho)=H^1_{:\Delta}(\sigma)$, and so only $H^1_{:\Delta}(\sigma)$ must be calculated. The extra effort this requires is mostly found in the proof of Lemma \ref{l.malquotient} below. We begin with the following elementary observation:

\begin{lemma}\label{l.wmix}
Let $\sigma:\Lambda\actson (X,\mu)$ and $\rho:\Lambda\actson (Y,\nu)$ be p.m.p. actions, and $\rho_0:\Lambda\actson (Y_0,\nu_0)$ a quotient of $\rho$ with quotient map $p_0:Y\to Y_0$. Let $\sigma\times\rho:\Lambda\actson (X\times Y,\mu\times\nu)$ be the diagonal product, $p_Y: X\times Y\to Y$ the projection onto $Y$, and let $p=p_0\circ p_Y$. Then $(\sigma\times\rho)\times_{p}(\sigma\times\rho)$ is isomorphic to
$$
(\sigma\times\sigma)\times(\rho\times_{p_0}\rho):\Lambda\actson((X\times X)\times (Y\times_{p_0} Y),\mu\times\mu\times(\nu\times_{p_0}\nu)).
$$
Thus, if $\sigma$ is weakly mixing and $\rho$ is weakly mixing relative to $p_0$, then the diagonal product $\sigma\times\rho$ is weakly mixing relative to $p$.
\end{lemma}

\begin{proof}
Let $\nu=\int\nu_{w} d\nu_0(w)$ be the disintegration of $\nu$ w.r.t. $p_0$. Then $\mu\times\nu=\int\mu\times\nu_wd\nu_0(w)$, and so
$$
(\mu\times\nu)\times_p(\mu\times\nu)=\int (\mu\times\nu_w)\times (\mu\times\nu_w)d\nu_0(w).
$$
It follows that the map $$
(X\times Y)\times_p (X\times Y)\to (X\times X)\times (Y\times_{p_0} Y): ((x,y),(x',y'))\mapsto ((x,x'),(y,y'))
$$ 
provides an isomorphism between $(\sigma\times\rho)\times_p(\sigma\times\rho)$ and $(\sigma\times\sigma)\times(\rho\times_{p_0}\rho)$.
\end{proof}

\subsection{The standard diagram} Consider now a countable discrete group $\Lambda$ with a normal subgroup $\Delta\lhd\Lambda$ and p.m.p. actions $\sigma:\Lambda\actson (X,\mu)$ and $\rho:\Lambda\actson (Y,\nu)$. Let $\sigma\times\rho:\Lambda\actson (X\times Y,\mu\times\nu)$ be the product action, and let $\bar\sigma_{\Delta}$, $\bar\rho_{\Delta}$ and $\overline{(\sigma\times\rho)}_{\Delta}$ be the corresponding actions of $\Lambda/\Delta$ on $X_\Delta$, $Y_\Delta$ and $(X\times Y)_\Delta$, respectively. It is in general \emph{not} the case that $\overline{(\sigma\times\rho)}_\Delta$ is the product of $\bar\sigma_\Delta$ and $\bar\rho_{\Delta}$, however the latter are quotients of $\overline{(\sigma\times\rho)}_\Delta$. To see this, let $p_{\sss Y}:X\times Y\to Y$ be the projection onto $Y$, and note that $p_{Y_\Delta}\circ p_Y:X\times Y\to Y_\Delta$ is $\Delta$-invariant. Thus $p^{\sss Y_\Delta}=p_{\sss Y_\Delta}\circ p_{\sss Y}$ is constant on almost all ergodic components of $\sigma\times\rho\restrict\Delta$, so it factors through $p_{\scriptscriptstyle (X\times Y)_{\Delta}}$ to $\bar p_{\sss Y_\Delta}: (X\times Y)_{\Delta}\to Y_\Delta$, and $\bar p_{\sss Y_\Delta}$ witnesses that $\bar\rho_{\Delta}$ is a quotient of $\overline{(\sigma\times\rho)}_\Delta$. Note that by definition the diagram 
\begin{equation*}\label{diagram}
\xymatrix{
X\times Y \ar[d]_{p^{\ }_{(X\times Y)_\Delta}} \ar[rr]^{p_Y^{\ }}\ar[drr]^(0.6){p^{Y_\Delta}}  & \ &  Y\ar[d]^{p^{\ }_{Y_\Delta}}\\
{(X\times Y)_{\Delta}} \ar[rr]_{\ \bar p^{\ }_{ Y_\Delta}} & \ & Y_\Delta
}
\end{equation*}
commutes. We will refer to this diagram as the \emph{standard diagram for $\sigma\times\rho$ relative to $\rho$ and $\Delta$}. The map $\bar p_{\scriptscriptstyle X_\Delta}:(X\times Y)_\Delta\to X_\Delta$ witnessing that $\bar\sigma_{\Delta}$ is a quotient of $\overline{(\sigma\times\rho)}_{\Delta}$ can be defined \emph{mutatis mutandis}, which together with $p_X$, $p_{X_\Delta}$ and $p_{(X\times Y)_\Delta}$ gives rise to the \emph{standard diagram for $\sigma\times\rho$ relative to $\sigma$ and $\Delta$}.

\begin{definition}
Suppose $\Delta\lhd\Lambda$ is a normal subgroup. We will say that $\sigma:\Lambda\actson X$ is weakly mixing (respectively malleable) relative to $\Delta$ if it is weakly mixing relative to the quotient $\sigma_\Delta$ (respectively malleable with respect to the quotient $\sigma_\Delta$.)
\end{definition}

% LEMMA: RELATIVE WEAKLY MIXING

\begin{lemma}\label{l.relweakmix}
With notation as in the standard diagram, if $\sigma$ is weakly mixing and $\rho$ is weakly mixing relative to $\Delta$ then $\overline{(\sigma\times\rho)}_\Delta$ is weakly mixing relative to $(\bar p_{\sss Y_\Delta},\bar\rho_{\Delta})$.
\end{lemma}

\begin{proof}
Since we have the sequence of extensions
$$
X\times Y\overset{p_{(X\times Y)_\Delta}}\longrightarrow (X\times Y)_\Delta\overset{\bar p_{\sss Y_\Delta}}\longrightarrow Y_\Delta
$$
it is enough to show that $\sigma\times\rho$ is weakly mixing relative to $p^{\sss Y_\Delta}$. This follows directly from Lemma \ref{l.wmix}.
\end{proof}

Note that if $\rho\restrict\Delta\times p_{\sss Y_\Delta}^{-1}(\bar y)$ is weakly mixing for $\nu_\Delta$-almost all $\bar y\in Y_\Delta$ then we have that $\rho$ is weakly mixing relative $p_{\sss Y_\Delta}$. Of course, in this case we also have that $(\sigma\times\rho)_{\Delta}$ is isomorphic to $\sigma_\Delta\times\rho_\Delta$.
\medskip

\subsection{Initial computation of $H^1_{:\Delta}$.} Consider again $\sigma$ and $\rho$ as above and the standard diagram. For ease of notation here and in the next lemma, let $Z=X\times Y$ and $\eta=\mu\times\nu$. Using that the standard diagram commutes, it follows easily that the map
\begin{equation}\label{eq.q}
q: ((x,y),(x',y'))\mapsto (p_{\sss (X\times Y)_\Delta}(x,y),p_{(X\times Y)_\Delta}(x',y'))
\end{equation}
of $Z\times_{p^{\sss Y_\Delta}}Z$ onto $Z_\Delta\times_{\bar p_{\sss Y_\Delta}} Z_\Delta$ witnesses that $\overline{(\sigma\times\rho)}_\Delta\times_{\bar p_{Y_\Delta}}\overline{(\sigma\times\rho)}_\Delta$ is an extension of $(\sigma\times\rho)\times_{p^{Y_\Delta}} (\sigma\times\rho)$.

% LEMMA: Malleability of ergodic quotient

\begin{lemma}\label{l.malquotient}
Let $\Delta\lhd\Lambda$, $\sigma$ and $\rho$ be as above, and let $p_{X_\Delta}: X\to X_\Delta$ be the canonical quotient map. Suppose there is a continuous function $[0,1]\to \Aut(X\times X,\mu\times\mu):t\mapsto S_t$ such that
\begin{enumerate}
\item $t\mapsto S_t$ witnesses that $\sigma$ is malleable and
\item $S_t$ commutes with the action $\Delta\times\Delta\actson X\times X: (\delta,\delta')\cdot (x,x')=(\delta\cdot_\sigma x,\delta'\cdot_\sigma x)$.
\end{enumerate}
Then $\overline{(\sigma\times\rho)}_\Delta$ is malleable relative to $\bar p_{Y_\Delta}$.
\end{lemma}

% PROOF: Malleability of the ergodic quotient

\begin{proof}
For $w\in X\times X$ we will write $w_0$ and $w_1$ for the components of $w$, i.e., $w=(w_0,w_1)$. For $t\in [0,1]$, define $f_t:Z\times_{p^{\sss Y_\Delta}}Z\to Z\times_{p^{\sss Y_\Delta}}Z$ by
$$
f_t((x,y),(x',y'))=((S_t(x,x')_0,y),(S_t(x,x')_1,y'))
$$
and 
$$
\tilde S_t(q((x,y),(x',y')))=q(f_t((x,y),(x',y'))),
$$
where $q$ is as in \eqref{eq.q}. To see that $\tilde S_t$ is a well-defined map on $Z_\Delta\times_{\bar p_{\sss Y_\Delta}}Z_\Delta$ to itself, note that if $\delta,\delta'\in\Delta$ then by (2) we have
$$
f_t(\delta\cdot_{\sss \sigma\times\rho} (x,y),\delta'\cdot_{\sss \sigma\times\rho} (x',y'))=(\delta\cdot_{\sss \sigma\times\rho} (S_t(x,x')_0,y),\delta'\cdot_{\sss \sigma\times\rho} (S_t(x,x')_1,y')).
$$
Thus $q\circ f_t$ is invariant under the action of $\Delta\times\Delta$, and so it factors to a map on $Z_\Delta\times_{\bar p_{Y_\Delta}} Z_\Delta$ (namely $\tilde S_t$) and so $\tilde S_t$ is well-defined. Moreover, since $f_t\in\Aut(Z\times_{p^{Y_\Delta}} Z,\eta\times_{p^{Y_\Delta}}\eta)$, it follows that $\tilde S_t\in \Aut(Z_\Delta\times_{\bar p_{Y_\Delta}} Z_\Delta, \eta_\Delta\times_{\bar p_{Y_\Delta}}\eta_\Delta)$, and since $t\mapsto f_t$ is continuous, so is $t\mapsto \tilde S_t$.

It remains to show that $t\mapsto \tilde S_t$ witnesses that $\overline{(\sigma\times\rho)}_\Delta$ is malleable. For this, let $\nu=\int\nu_{\bar y}^\Delta d\nu_\Delta(\bar y)$ be the disintegration of $\nu$ w.r.t. $p_{Y_\Delta}$ and $\eta=\int \eta^{\Delta}_{\bar z} d\eta_\Delta(\bar z)$ be the disintegration of $\eta$ w.r.t. $p_{Z_\Delta}$. As this coincides with the ergodic decomposition of $\rho\restrict\Delta$, the measure $\nu_{\bar y}^\Delta$ is $\Delta$-invariant and uniquely $\rho\restrict\Delta$-ergodic on $p_{Y_\Delta}^{-1}(\bar y)$ for $\nu_\Delta$-almost all $\bar y\in Y_\Delta$. Similarly, $\eta_\Delta$-almost all $\eta^\Delta_{\bar z}$ are $\sigma\times\rho\restrict\Delta$-ergodic. Thus for $\eta_\Delta$-a.a. $\bar z\in Z_\Delta$ the push-forward measure $p_Y^*[\eta^\Delta_{\bar z}]$ is $\rho\restrict\Delta$-invariant and ergodic and supported on $p_{Y_\Delta}^{-1}(\bar p_{Y_\Delta}(\bar z))$, and so by unique ergodicity we have $p^*_Y[\eta^\Delta_{\bar z}]=\nu^\Delta_{\bar p_{Y_\Delta}(\bar z)}$ for $\eta_\Delta$-a.a. $\bar z\in Z_\Delta$. Let $\eta_\Delta=\int\eta_{\Delta,\bar y} d\nu_\Delta(\bar y)$ be the disintegration of $\eta_\Delta$ w.r.t. $\bar p_{Y_\Delta}$. It follows that for $\nu_\Delta$-a.a. $\bar y$ the set
$$
A_{\bar y}=\{\bar z\in Z_\Delta: \bar p_{Y_\Delta}(\bar z)=\bar y\wedge p^*_Y[\eta_{\bar z}^\Delta]=\nu^\Delta_{\bar y}\}
$$
has full $\eta_{\Delta,\bar y}$ measure. Then for $\bar y$ such that $\eta_{\Delta,\bar y}(A_{\bar y})=1$ and $\bar z,\bar z'\in A_{\bar y}$ we have 
$$
\nu_{\bar y}^\Delta(\proj_{Y}(p_{Z_\Delta}^{-1}(\bar z)))=\nu_{\bar y}^\Delta(\proj_{Y}(p_{Z_\Delta}^{-1}(\bar z')))=1
$$
In particular, there is $y_0\in \proj_{Y}(p_{Z_\Delta}^{-1}(\bar z))\cap \proj_{Y}(p_{Z_\Delta}^{-1}(\bar z'))$. Let $x,x'\in X$ be such that $p_{Z_\Delta}(x,y_0)=\bar z$ and $p_{Z_\Delta}(x',y_0)=\bar z'$. Then 
$$
\tilde S_1(\bar z,\bar z')=q\circ f_1((x,y_0),(x',y_0))=q((x',y_0),(x,y_0))=(\bar z',\bar z),
$$
and so $t\mapsto\tilde S_t$ witnesses that $(\sigma\times\rho)_\Delta$ is malleable relative to $\bar p_{Y_\Delta}$.
\end{proof}

% LEMMA: Bernoulli quotient malleability.

\begin{lemma}\label{l.malleable}
Let $\Lambda$ be a countable discrete group, $\Delta\lhd\Lambda$ a normal subgroup, and $(X_0,\mu_0)$ a non-atomic standard Borel probability space. Let $\sigma_0:\Lambda\actson\N$ be an action, let $\sigma:\Lambda\actson X_0^\N$ be the generalized Bernoulli shift, and let $\rho:\Lambda\actson (Y,\nu)$ be any p.m.p. action. Then the action $\overline{(\sigma\times\rho)}_\Delta$ is malleable relative to $\bar p_{Y_\Delta}$.
\end{lemma}

% PROOF: Bernoulli quotient malleability

\begin{proof}
We may of course assume that $X_0=[0,1)$ and $\mu_0$ is the Lebesgue measure. Further, w.l.o.g., all $\sigma_0\restrict\Delta$ are finite. Let $X=X_0^\N=[0,1)^\N$, $\mu=\mu_0^\N$. For $x\in X_0^{\N\times\{0,1\}}$ and $i\in\{0,1\}$, write $x_i$ for the element in $X$ defined by $x_i(n)=x(n,i)$, for all $n\in\N$, so that $x\mapsto (x_0,x_1)$ canonically identifies $X^2$ and $X_0^{\N\times\{0,1\}}$. Define $[0,1)\to\Aut(X^{\N\times\{0,1\}},\mu^{\N\times\{0,1\}}):t\mapsto S_t$ by
$$
S_t(x)(n,i)=\left\{\begin{array}{ll}
x(n,i-1) & \text{ if for all } \delta\in\Delta,\ x(\sigma_0(\delta)(n),0),x(\sigma_0(\delta)(n),1)<t\\
x(n,i) & \text{ otherwise.}
\end{array}\right.
$$
Note that since $\Delta$ is a normal subgroup in $\Lambda$, $S_t$ commutes with the product action $\sigma\times\sigma$, and by definition $S_t$ commutes with the action of $\Delta\times\Delta$ on $X^{\N\times\{0,1\}}$. Thus Lemma \ref{l.malquotient} applies.
\end{proof}

We conclude:

% LEMMA: Computation of relative cohomology group, I.

\begin{corollary}\label{l.relcoho}
Let $\Lambda$, $\Delta$, $\sigma_0$ and $\sigma$ be as in the previos lemma, and suppose that $\sigma_0$ has infinite orbits so that $\sigma$ is weakly mixing. Suppose further that $\Lambda/\Delta$ has strongly $\{\T\}$-cocycle superrigid malleable weakly mixing actions. Then for any p.m.p. action $\rho:\Lambda\actson (Y,\nu)$ such that $\sigma\times\rho$ is weakly mixing relative to $p^{Y_\Delta}$ we have $H^1(\overline{(\sigma\times\rho)}_\Delta)\simeq H^1(\bar\rho_\Delta)$ and $H^1_{:\Delta}(\sigma\times\rho)\simeq H^1_{:\Delta}(\rho)$. In particular this holds for any $\rho$ which is weakly mixing relative to $p_{Y_\Delta}$.
\end{corollary}

% Ergodic decompositions of generalized Bernoulli actions.

\subsection{Ergodic decompositions of generalized Bernoulli actions.} Before proceeding with further calculations of $H^1_{:\Delta}$, we pause briefly to give a description of the ergodic decomposition of a generalized Bernoulli shift.

Let $X_0$ be a compact Polish space equipped with a standard Borel probability measure $\mu_0$. Let $\sigma_0:\Delta\actson\N$ be an action, and let $\sigma:\Delta\actson X_0^\N$ be the Bernoulli action, $\mu=\mu_0^\N$ the product measure on $X=X_0^\N$. Also, for $B\subseteq\N$ which is $\sigma_0$-invariant, let $\sigma_B:\Delta\actson X_0^B$ be the Bernoulli action, $\mu_B=\mu_0^B$, and let $p_B:X_0^\N\to X_0^B$ denote the projection map. Denote by $E_{\sigma}$ and $E_{\sigma_B}$ the induced orbit equivalence relations. Define in $X^\N$  an equivalence relation
\begin{equation}\label{eq.sim}
x\sim y\iff \text{ for all $\sigma_0$-invariant and finite $B\subseteq \N$ we have $p_B(x)E_{\sigma_B} p_B(y)$.}
\end{equation}
Then $\sim$ is a closed equivalence relation, and therefore the quotient $X/\!\!\sim$ is standard (recall that  $X_0$ is compact.) It is clear that the $\sim$-classes $\sigma$-invariant. Let $Y=X_0^\N/\!\!\sim$ and let $\pi:X\to Y: x\mapsto [x]_\sim$ be the quotient map, $\nu=\pi^*[\mu]$ the push-forward measure on $Y$.

For the purpose of the next Lemma, assume now that all $\sigma_0$-classes are finite. Evidently the action of $\sigma$ on $[x]_{\sim}$ is profinite, and so there is a unique measure $\mu_{[x]_\sim}$ on $[x]_\sim$ induced by giving $p_B([x]_\sim)$ the normalized counting measure, for each $B\subseteq\N$ finite and invariant. The measure $\mu_{[x]_\sim}$ is $\sigma$-invariant and it is ergodic since $p_B^*[\mu_{[x]_\sim}]$ is ergodic for $\sigma_B$.

% LEMMA: Ergodic decomposition of generalized Bernoulli actions

\begin{lemma}\label{l.ergdec.v1}
With notation as in the preceding paragraph, the maps $\pi:X\to Y$ and $Y\to \mathcal{EI}_{\sigma}:y\mapsto\mu_y$ provides an ergodic decomposition of $\sigma:\Delta\actson X^\N$.
\end{lemma}

% PROOF: Erg. decomp. Lemma.

\begin{proof}
It suffices to show that the disintegration identity
\begin{equation}\label{l.ergdec.eq}
\mu=\int\mu_y d\nu(y)
\end{equation}
holds. For this, let $B\subseteq\N$ be finite and $\sigma_0$-invariant. It is clear that $X_0^B/E_{\sigma_B}$ is standard, and that if we let $\pi_B: X_0^B\to X_0^B/E_{\sigma_B}$ be the quotient map and $\nu_B=\pi_B^*[\mu_B]$, then the disintegration identity
$$
\mu_B(A)=\int \frac{|[x]_{E_{\sigma_B}}\cap A|}{|[x]_{E_{\sigma_B}}|} d\nu_B([x]_{E_{\sigma_B}})
$$
holds. Further, note that by definition of $\sim$ the map $\hat p_B:[x]_\sim\mapsto [x]_{E_{\sigma_B}}$ from $X/\!\!\sim$ to $X_0^B/E_{\sigma_B}$ is well-defined and that $\pi_B\circ p_B= \hat p_B\circ \pi$. Thus $\nu_B=\hat p_B^*[\nu]$.

Now fix $C_i\subseteq X_0$ measurable for each $i\in B$, and consider the cylinder set
$$
C=\{x\in X_0^\N: (\forall i\in B) x(i)\in C_i\}.
$$
For such a cylinder set we have that
$$
\mu_{[x]_\sim}(C)=\frac{|[p_B(x)]_{E_{\sigma_B}}\cap \prod_{i\in B} C_i|}{|[p_B(x)]_{E_{\sigma_B}}|},
$$
and since the right hand side only depends $[p_B(x)]_{E_{\sigma_B}}=\hat p_B([x]_\sim)$ we have
$$
\int\mu_y(C)d\nu(y)=\int \frac{|[x]_{E_{\sigma_B}}\cap \prod_{i\in B} C_i|}{|[x]_{E_{\sigma_B}}|}d\nu_B([x]_{E_{\sigma_B}})=\mu_B(C)=\mu(C)
$$
as required.
\end{proof}

Dropping again the assumption that all $\sigma_0$-classes are finite, let $\N=C_0\cup C_1$ where $C_0,C_1$ are $\sigma_0$-invariant, and $C_0$ consists of all $n\in\N$ with $[n]_{\sigma_0}$ finite. Letting $\sigma_{C_i}:\Delta\actson X_0^{C_i}$, ($i=0,1$), be the corresponding Bernoulli shifts, we have $\sigma=\sigma_{C_0}\times\sigma_{C_1}$. Note that $p_{C_1}([x]_{\sim})=X_0^{C_1}$, while $p_{C_0}([x]_{\sim})$ is an ergodic component of $\sigma_{C_0}$ with ergodic invariant measure $\mu_{p_{C_0}([x]_\sim)}$ as described above. Since $[x]_{\sim}=p_{C_0}([x]_\sim)\times X_0^{C_1}$, we conclude:

% LEMMA: Ergodic decomposition of generalized Bernoulli action

\begin{lemma}\label{l.ergdec}
The maps
$$
\pi:X_0^\N\to Y: x\mapsto [x]_\sim,\ Y\to\mathcal{IE}_{\sigma}: y\mapsto \mu_{p_{C_0}([x]_\sim)}\times\mu_0^{C_1}
$$ 
provide an ergodic decomposition of $\sigma:\Delta\actson X_0^\N$.
\end{lemma}

% SUBSECTION

\subsection{Quotients of generalized Bernoulli shifts.} Quotients of Bernoulli shifts were introduced in \cite{popa06} where they were used to obtain an infinite family of non-orbit equivalent actions of a countable group with the relative property (T) over a weakly normal subgroup. These actions were distinguished by their first 1-cohomology group, which was computed using a forerunner to Theorem \ref{popa_csr}. These ideas were then developed further in \cite{popavaes08}. We will now use a similar construction to obtain a large family of measure preserving actions for which the relative 1-cohomology group is calculable.

Let $A$ be a (non-trivial) second countable Abelian compact group, and let $\lambda_A$ be the normalized Haar measure on $A$. The group $A$ acts on $X=A^\N$ by rotation, i.e., by $(a\cdot x)(i)=ax(i)$, where $a\in A$, $x\in A^\N$ and $i\in\N$, and this action preserves the product measure $\mu=\lambda_A^\N$. Let $[x]_A$ denote the $A$-orbit of $x\in A^\N$, let $X_A=X/A=\{[x]_A:x\in X\}$, and let $p_A:X\to X_A:x\mapsto [x]_A$ be the quotient map. Note that since $A$ is compact the space $X_A$ is standard in the quotient Borel structure. We let $\nu_A=p_A^*[\mu]$. Let $\Lambda$ be a countable discrete group and let $\sigma_0:\Lambda\actson\N$ be an action. The induced Bernoulli action $\sigma:\Lambda\actson A^\N$ clearly commutes with the action of $A$ on $A^\N$, and so $\sigma$ induces an action of $\Lambda$ on $X_A$. We call this action the \emph{quotient} of $\sigma$ by $A$, and denote it by $\sigma_A$. Our goal now is to establish the following ``quotient version'' of Lemma \ref{l.relcoho}:

% LEMMA: RELATIVE COHOM FOR QUOTIENTS

\begin{theorem}\label{t.cohoquotient}
Let $\Lambda$ be a countable discrete group with a normal subgroup $\Delta\lhd\Lambda$, and let $A$ be an infinite compact second countable Abelian group. Suppose $\Lambda/\Delta$ has strongly $\{\T\}$-cocycle superrigid weakly mixing malleable actions. Suppose further that there is a $\Delta$-suitable action $\sigma_0:\Lambda\actson\N$, and let $\sigma:\Lambda\actson A^\N$ be the corresponding generalized Bernoulli shift, $\sigma_A$ its quotient. Then for any p.m.p. action $\rho:\Lambda\actson Y$ such that $\sigma\times\rho$ is weakly mixing with respect to $p^{Y_\Delta}$ we have $H^1_{:\Delta}(\sigma_A\times\rho)\simeq H^1_{:\Delta}(\rho)\times \Char(A)$. In particular, $H^1_{:\Delta}(\sigma_A)\simeq\Char(\Lambda)\times \Char(A)$.
\end{theorem}

% SUBSECTION

\subsection{The actions $(\sigma_A)_\Delta$ and $(\sigma_A\times\rho)_\Delta$} Before proving Theorem \ref{t.cohoquotient} we shall give a description of $(\sigma_A)_\Delta$ and $(\sigma_A\times\rho)_\Delta$. Let $\Delta\lhd\Lambda$ be countable discrete groups, and let $A$ an infinite second countable compact Abelian group. Let $\sigma_0:\Lambda\actson\N$ be an action, $\sigma:\Lambda\actson A^\N$ the induced Bernoulli action on $X=A^\N$, equipped with product measure. Let $\sim$ be defined as in \eqref{eq.sim}, and let $\pi: X\to X/\!\!\sim:x\mapsto [x]_\sim$ and $x\mapsto \mu_{[x]_\sim}$ be defined as before. We identify $X/\!\!\sim$ with $X_\Delta$, and $\pi$ with $p_{X_\Delta}$. Note then that $A$ acts on $X/\!\!\sim$ by $a\cdot [x]_\sim=[a\cdot x]_\sim$, and this is well-defined since the action of $A$ on $X$ commutes with $\sigma$. The action of $A$ on $X_\Delta$ is free on a set of measure one. This follows since $A$ acts freely on the set
$$
\{[x]_\sim: (\forall i,j,k,l\in\N) i\neq j\implies x(i)x(k)^{-1}=x(j)x(l)^{-1}\}
$$
which is invariant and has full measure. Because $A$ is compact the quotient $X_{\Delta,A}=X_\Delta/A$ is standard, and we equip it with the push-forward measure; we let $q_A:X_\Delta\to X_{\Delta,A}$ be the quotient map. The action of $A$ on $X_\Delta$ commutes with the action $\sigma_\Delta$ and so the action $\sigma_\Delta$ induces a p.m.p. action on $X_{\Delta,A}$ which we denote by $\sigma_{\Delta,A}$.

Let now $\rho:\Lambda\actson (Y,\nu)$ by any probability measure preserving action. The group $A$ acts on the space $X\times Y$ in the first coordinate. This action commutes with $\sigma\times\rho$, and we let $(X\times Y)_A=(X\times Y)/A$ be the quotient space (which is standard) and let $p_A^Y:X\times Y\to (X\times Y)_A$ be the quotient map. For ease of notation, we will sometimes write $[z]_A$ for $p_A^Y(z)$. Then $\sigma\times\rho$ induces an action $(\sigma\times\rho)_A\actson (X\times Y)_A$ which preserves the push-forward measure. On the other hand, the action of $A$ permutes the ergodic components of $\sigma\times\rho\restrict\Delta$, and so induces an action of $A$ on $(X\times Y)_\Delta$ defined by $a\cdot p_{(X\times Y)_\Delta}(x,y)=p_{(X\times Y)_\Delta}(a\cdot x,y)$. Since the standard diagram 
\begin{equation*}\label{diagram}
\xymatrix{
X\times Y \ar[d]_{p^{\ }_{(X\times Y)_\Delta}} \ar[rr]^{p_X^{\ }}\ar[drr]^(0.6){p^{X_\Delta}}  & \ &  X\ar[d]^{p^{\ }_{X_\Delta}}\\
{(X\times Y)_{\Delta}} \ar[rr]_{\ \bar p^{\ }_{ X_\Delta}} & \ & X_\Delta
}
\end{equation*}
for $\sigma\times\rho$ relative to $\sigma$ and $\Delta$ commutes, the action on $A$ on $(X\times Y)_\Delta$ is free almost everywhere since the action of $A$ on $X_\Delta$ is. As before, since $A$ is compact the quotient $(X\times Y)_{\Delta,A}=(X\times Y)_\Delta/A$ by this action is a standard space which we give the push-forward measure, denoted $(\mu\times\nu)_{\Delta,A}$; we let $q_A^Y:(X\times Y)_\Delta\to (X\times Y)_{\Delta,A}$ be the quotient map. Finally, the action of $A$ on $(X\times Y)_\Delta$ commutes with the action $(\sigma\times\rho)_\Delta$ and so the action $(\sigma\times\rho)_\Delta$ induces a p.m.p. action on $(X\times Y)_{\Delta,A}$ which we denote by $(\sigma\times\rho)_{\Delta,A}$.

% LEMMA

\begin{lemma}\label{l.isoquotients}
We have $((\sigma\times\rho)_A)_\Delta\simeq (\sigma\times\rho)_{\Delta,A}$ and, in particular, $(\sigma_A)_\Delta\simeq \sigma_{\Delta,A}$.
\end{lemma}

% PROOF

\begin{proof}
Consider the quotient map $p_{\Delta,A}=q_A^Y\circ p_{(X\times Y)_\Delta}$. Define a map $q_\Delta: (X\times Y)_A\to (X\times Y)_{\Delta,A}$ by $q([x]_A)=p_{\Delta,A}(x)$ and note that $q_\Delta$ is well-defined since $p_{A,\Delta}$ is $A$-invariant. Moreover, 
\begin{align*}
q_\Delta((\sigma\times\rho)_A(\gamma)([z]_A))&=q_\Delta(p_A^Y(\sigma\times\rho(\gamma)(z)))
=p_{\Delta,A}((\sigma\times\rho)(\gamma)(z))\\
&=(\sigma\times\rho)_{\Delta,A}(\gamma)(p_{\Delta,A}(z))=(\sigma\times\rho)_{\Delta,A}(\gamma)(q_\Delta([z]_A))
\end{align*}
for all $\gamma\in\Lambda$, thus $(\sigma\times\rho)_{\Delta,A}$ is an extension of $(\sigma\times\rho)_A$, and we have the following commutative diagram of extensions:
\begin{equation}\label{diagram2}
\xymatrix{
\  & X\times Y\ar[dl]_{p_{(X\times Y)_\Delta}\ }\ar[dd]^{p_{\Delta,A}}\ar[dr]^{p_A^Y} & \ \\
(X\times Y)_\Delta\ar[dr]_{q_A^Y} & \ & (X\times Y)_A\ar[dl]^{q_\Delta^{\ }}\\
\ & (X\times Y)_{\Delta,A} & \
}
\end{equation}
To prove that $((\sigma\times\rho)_A)_\Delta\simeq (\sigma\times\rho)_{\Delta,A}$ it suffices to show that the fibre $q_\Delta^{-1}(z)$ is an ergodic component of $(\sigma\times\rho)_A\restrict\Delta$ for almost all $z\in (X\times Y)_{\Delta,A}$. For this, simply note that since $A$ acts freely on $(X\times Y)_\Delta$ we have $p_{\Delta,A}^{-1}(z)=A\cdot p_{(X\times Y)_\Delta}^{-1}(z')$ for any $z'$ such that $q^Y_A(z')=z$. It is therefore clear that $p_A^Y(p_{\Delta,A}^{-1}(z))=q_\Delta^{-1}(z)$ is an ergodic component of $(\sigma\times\rho)_A\restrict\Delta$.
\end{proof}

From now on we identify $((X\times Y)_A)_{\Delta}$ and $(X\times Y)_{\Delta,A}$, as well as $((\sigma\times\rho)_A)_\Delta$ and $(\sigma\times\rho)_{\Delta, A}$.

\begin{proof}[Proof of Theorem \ref{t.cohoquotient}]

The proof is similar to the proof of \cite[Lemma 2.10]{popa06} (see also \cite[Lemma 3.2]{tornquist11}.) For ease of notation, elements of $\Lambda/\Delta$ are indicated with a bar, e.g. $\bar\gamma$. Let $\alpha\in Z^1(\overline{((\sigma\times\rho)_{A})}_\Delta)$, and define $\hat\alpha\in Z^1(\overline{(\sigma\times\rho)}_\Delta)$ be defined by
$$
\hat\alpha(\bar\gamma,w)=\alpha(\bar\gamma,q_A^Y(w)).
$$
Then by Lemma \ref{l.relcoho} there is a function $f:(X\times Y)_{\Delta}\to \T$ and $\vartheta\in Z^1(\bar\rho_{\Delta})$ such that 
$$
\hat\alpha(\bar\gamma,w)=f(\overline{(\sigma\times\rho)}_\Delta(\bar\gamma)(w))\vartheta(\bar\gamma,\bar p_{Y_{\Delta}}(w))f(w)^{-1}.
$$
Note then that the 1-cocycle $\hat\alpha(\bar\gamma,w)\theta(\bar\gamma,\bar p_{Y_\Delta}(w))^{-1}$ is $A$-invariant, and so the 1-coboundary 
$$
(\bar\gamma,w)\mapsto f(\overline{(\sigma\times\rho)}_\Delta(\bar\gamma)(w))f(w)^{-1}
$$ is $A$-invariant. The $A$-invariance gives that for all $a\in A$ we have
$$
f(\overline{(\sigma\times\rho)}_\Delta(\bar\gamma)(a\cdot w))f(\overline{(\sigma\times\rho)}_\Delta(\bar\gamma)(w))^{-1}=f(a\cdot w)f(w)^{-1},
$$
which proves that the function $\zeta_a:w\mapsto f(a\cdot w)f(w)^{-1}$ is $\overline{(\sigma\times\rho)}_\Delta$-invariant for each $a\in A$. Since $(\sigma\times\rho)_\Delta$ is ergodic it follows that $\zeta_a$ is constant a.e., say $\zeta_a(w)=\chi(a)\in\T$ for all $a\in A$. Thus $f(a\cdot w)=\chi(a) f(w)$, and so $f$ is an $A$-eigenfunction with $\chi\in\Char(A)$ the associated character of $A$. We conclude that any $\hat\alpha$ as above can be written as a product of a 1-cocycle of $\bar\rho_\Delta$ and a 1-coboundary arising from an $A$-eigenfunction $f:(X\times Y)_\Delta\to \T$.

On the other hand, if $f:(X\times Y)_\Delta\to \T$ is an $A$-eigenfunction, then the corresponding 1-coboundary for $\overline{(\sigma\times\rho)}_\Delta$ is $A$-invariant, and so gives rise to a 1-cocycle of $\overline{(\sigma\times\rho)}_{\Delta,A}$. Moreover, it is easy to see that any two 1-cocycles in $Z^1(\overline{(\sigma\times\rho)}_{\Delta,A})$ arising in this fashion are cohomologous in $Z^1(\overline{(\sigma\times\rho)}_{\Delta,A})$ precisely when their associated characters are the same. At the same time, it is clear that for any character $\chi\in\Char(A)$ there is a $A$-eigenfunction $f:(X\times Y)_\Delta\to \T$ such that $f(a\cdot w)=\chi(a)f(w)$. Thus we have $H^1(\overline{(\sigma\times\rho)}_{\Delta,A})\simeq H^1(\bar\rho_{\Delta})\times\Char(A)$. This together with Lemma \ref{l.canisom} and Lemma \ref{l.isoquotients} gives $H^1_{:\Delta}((\sigma\times\rho)_A)\simeq H^1_{:\Delta}(\rho)\times\Char(A)$ as required.
\end{proof}

We conclude this section by stating the following corollary which is the main result that will be used going forward. For notational simplicity, we let $\hat A$ stand for $\Char(A)$.

\begin{corollary}\label{t.pseudofamily}
Let $\Gamma$ be a countable discrete group, $\Delta\lhd\Lambda\leq\Gamma$ subgroups such that $\Lambda/\Delta$ has strongly $\{\T\}$-cocycle superrigid wearkly mixing malleable actions. Suppose there is an action $\sigma_0:\Gamma\actson\N$ such that $\sigma_0\restrict \Lambda$ is $\Delta$-suitable. Then for any $\rho:\Gamma\actson (Y,\nu)$ be any \emph{a.e. free} weakly mixing action such that $\rho\restrict\Delta$ is ergodic. Then the family
$$
\langle\sigma_{\hat A}\times\rho:A \text{ is discrete countable torsion-free Abelian}\rangle
$$ 
is a family of a.e. free, ergodic p.m.p. actions such for all countable torsion-free Abelian groups $A_0$ and $A_1$ it holds that $A_0$ is isomorphic to $A_1$ if and only if $\sigma_{\hat A_0}\times\rho$ is conjugate to $\sigma_{\hat A_1}\times\rho$.
\end{corollary}

\begin{proof}
It is clear that if $A_0\simeq A_1$ then $\sigma_{\hat A_0}\times\rho$ is conjugate to $\sigma_{\hat A_1}\times\rho$. For the other direction, note that clearly $(\sigma\times\rho)_{\hat A}\simeq\sigma_{\hat A}\times\rho$, and so it follows from Theorem \ref{t.cohoquotient} that $H_{:\Delta}(\sigma_{\hat A}\times\rho)\simeq \hat A\times H^1_{:\Delta}(\rho)$. Since $\rho\restrict\Delta$ is ergodic it follows from \ref{l.ergtrivial} that $H^1_{:\Delta}((\sigma_{\hat A}\times\rho)_{\Delta})\simeq \hat A\times \Char(\Lambda/\Delta)$. Since $\Char(\Lambda/\Delta)$ is compact and the topological group $H^1_{:\Delta}$ is a conjugacy invariant of $\sigma_{\hat A}\times\rho$ we now have that $\sigma_{\hat A_0}\times\rho\conj\sigma_{\hat A_1}\times\rho$ implies that $A_0\simeq A_1$, as required.
\end{proof}

% SECTION: BOREL MANY-ONE

\section{A Stern absoluteness argument}

The present section is dedicated to a technical result on many-to-one Borel reductions. The main aim is to prove Theorem \ref{t.tfa} below, which is a consequence of results of Harrington and Hjorth. The argument uses a metamathematical technique developed by Stern in \cite{stern84}, known now as \emph{Stern's forcing absoluteness}, which we briefly review below, and which requires some basic knowledge of forcing (see e.g. \cite{kunen80}) and effective descriptive set theory as found in e.g. \cite[Ch. 2]{kanovei08}. Since the metamathematical techniques will not play any role in subsequent parts of this paper, the reader who wishes to do so may treat Theorem \ref{t.tfa} below as a ``black box'' and move on to the next section.

\medskip

Recall from \S 2 that ${\bf TFA}$ denotes the standard Borel space of countable torsion free Abelian groups with underlying set $\N$, and that $\simeq^{\bf TFA}$ denotes the isomorphism relation on ${\bf TFA}$. Recall also that the notion of a Borel countable-to-1 reduction and the notation $\leq_B^\N$ was defined \S 2.

\begin{theorem}\label{t.tfa}
Let $E$ be an equivalence relation on a standard Borel space $Y$, and suppose $\simeq^{\bf TFA}\leq_B^\N E$. Then $E$ is not Borel.
\end{theorem}

\subsection{Borel many-to-1 reducibility.}

The following Theorem is due to Harrington (unpublished, see \cite{hjorth98}; for a proof that does not use metamathematics, see \cite{hjkelo98}.)

\begin{theorem}[Harrington]\label{t.harrington}
There is a family of Borel equivalence relations $E_\alpha$, $\alpha<\omega_1$, such that for no Borel equivalence relation $E$ do we have $E_\alpha\leq_B E$ for all $\alpha<\omega_1$.
\end{theorem}

Let us make the following definition:

\begin{definition}
Let $E$, $F$  and $R$ be equivalence relations on Polish spaces $X$, $Y$ and $Z$, respectively. We will say that $E$ is \emph{$R$-to-one reducible} to $F$, written $E\leq_B^R F$, if there is a Borel functions $f:X\to Y$ such that
$$
(\forall x_1,x_2\in X) x_1 E x_2\implies f(x_1) F f(x_2),
$$
and for each $F$-class $[y]_F$ there is a Borel function $g_{[y]_F}: X\to Z$ such that
\begin{equation}\label{eq.gy}
(\forall x_1,x_2\in f^{-1}([y]_F)) x_1 E x_2\iff g(x_1) R g(x_2).
\end{equation}
We will say that $E$ is \emph{unformly} $R$-to-one reducible to $F$ if there is $f$ as above and a 
\emph{single} Borel $g:X\to Z$ such that $g_{[y]_F}=g$ satisfies \eqref{eq.gy} for all $y\in Y$.
\end{definition}

Fix $E$, $F$ and $R$ as above, and define $F\times R$ on $Y\times Z$ by
$$
(y,z) F\times R (y',z')\iff yFy'\wedge zRz'
$$
It is then clear that we have $E\leq_B^R F$ uniformly precisely when $E\leq_B F\times R$. Thus from Harrington's theorem we have
\begin{corollary}\label{c.harrington}
Let $F$ and $R$ be equivalence relations on Polish spaces, and let $E_\alpha$, $\alpha<\omega_1$ be as in Theorem \ref{t.harrington}. If for all $\alpha<\omega_1$ we have $E_\alpha\leq_B^R F$ uniformly then either $F$ or $R$ is not Borel.
\end{corollary}

The aim below is to prove a version of Corollary \ref{c.harrington} without the uniformity assumption in the special case when $R$ is equality on $\N$. This will follow from a result of Hjorth once we prove a slight strengthening of a result due to Harrington, which we now state. We first need several definitions. (Our notation roughly follows that of \cite{hjorth98} with a few minor changes.)

% Note to the typesetter: We use \epsilon and not \epsilon and not $\in$ on purpose in the next few lines. 
% Please don't change that, even if it is standard editorial policy.
% There is a mathematical reason behind the usage!

\begin{definition}\label{d.epsstruct}
Let $\epsilon$ be a binary relation on $\N$. We say that $(\N,\epsilon)$ is an \emph{$\epsilon$-structure} if
\begin{enumerate}
\item $(\forall m)(\exists n) m\inx{} n\vee n\inx{} m$.
\item $(\N,\epsilon)$ is extensional, that is, for all $m,n\in\N$ if
$$
\{k\in\N: k\inx{} m\}=\{k\in\N: k\inx{} n\}
$$
then $m=n$.
\item $\epsilon$ is wellfounded.
\end{enumerate}
Two $\epsilon$-structures $(\N,\epsilon_0)$ and $(\N,\epsilon_1)$ are isomorphic if there is a permutation $\sigma:\N\to\N$ such that for all $n,m\in\N$ we have $m\inx{0} n$ if and only if $\sigma(m)\inx{1}\sigma(n)$.
\end{definition}

For $x\in 2^{\N\times\N}$, let $\epsilon_x$ denote the binary relation on $\N$ defined by
$$
m\inx{x} n\iff x(m,n)=1.
$$
We let $B$ denote the set of $x\in 2^{\N\times\N}$ such that $(\N,\epsilon_x)$ is an $\epsilon$-structure, and define 
$$
x\simeq y\iff (\N,\epsilon_x)\text{ is isomorphic to } (\N,\epsilon_y).
$$
For $\alpha<\omega_1$, let $B_\alpha\subseteq B$ be the set of $x\in B$ such that $\epsilon_x$ has rank at most $\alpha$. Then $B_\alpha$ is Borel, and it can be shown (see \cite{hjorth98}) that so is $\simeq^\alpha=\simeq\restrict B_\alpha$, the restriction of $\simeq$ to $B_\alpha$.

Harrington proved Theorem \ref{t.harrington} by showing that if $E$ is an equivalence relation on a standard Borel space and $\simeq^\alpha\leq_B E$ for all $\alpha<\omega_1$, then $E$ is not Borel. Here we prove the following extension of this:

\begin{theorem}\label{t.ctblto1}
Let $E$ be an equivalence relation on a standard Borel space. If $\simeq^\alpha\leq_B^{\N} E$ for all $\alpha<\omega_1$, then $E$ is not Borel.
\end{theorem}

\subsection{Stern's absoluteness} We now briefly describe the key elements of Stern's forcing absoluteness. We only give enough detail that we can state the main lemma that we need. The reader should consult \cite{stern84} or \cite{hjorth98} for the proof of this lemma; for forcing basics we refer to \cite{kunen80}.

We will use the notations related to trees on $\N$ introduced in \S2.4. Per the usual conventions, for $s\in\N^{<\N}$ we let
$$
N_s=\{t\in\N^\N: t\supset s\},
$$
be the \emph{basic open neighborhood in $\N^\N$ determined by $s$}.

\begin{definition}[See \cite{hjorth98} Definition 1.1 and Definition 1.2]\ 

1. A \emph{Borel code} is a pair $(T,f)$ consisting of a wellfounded tree $T\subseteq \N^{<\N}$ and a function 
$$
f:\{t\in T: t\text{ is terminal}\}\to \N^{<\N}.
$$
Fix a Borel code $(T,f)$. We define for all $s\in T$ the Borel sets $B(s,T,f)$ by backwards recursion as follows: If $s\in T$ is a terminal note then we let $B(s,T,f)=N_{f(s)}$. If $B(t,T,f)$ has been defined for all $s\subset t\in T$ then we let
$$
B(s,T,f)=\bigcap\{\N^\N\setminus B(s^\frown n,T,f):s^\frown n\in T\}.
$$
The \emph{Borel set coded by $(T,f)$} is the set $B(\emptyset, T,f)$. The \emph{rank} of the Borel code $(T,f)$ is the rank of $\emptyset\in T$ (in the sense of \cite[2.E]{kechris95}.) One can easily prove that if $(T,f)$ has rank $\gamma<\omega_1$ then $B(\emptyset, T,f)$ is a $\mathbf{\Pi}^0_\gamma$-set, and that any $\mathbf{\Pi}^0_\gamma$-set in $\N^\N$ has a code of rank $\gamma$.

2. A \emph{virtual Borel set} is a triple $(\P,p,\tau)$ consisting of a poset $\P$ (in the sense of \cite{kunen80}), a condition $p\in\P$, and a $\P$-name $\tau$ such that
$$
p\forces_{\P} \text{``$\tau$ is a Borel code''}.
$$
We say that $(\P,p,\tau)$ is a \emph{virtual Borel set of rank $\gamma$}, for some $\gamma\in\mathbb{ON}$, if
$$
p\forces_{\P}\text{``$\tau$ is a Borel code of rank $\gamma$''}.
$$
Note that $\gamma$ may not be a countable ordinal in the ground model.

3. If $\P$ is a poset and $G$ is a filter on $\P\times\P$ then we let
$$
G_r=\{p\in\P: (\exists q\in\P) (p,q)\in G\}
$$
and
$$
G_l=\{q\in\P: (\exists p\in\P) (p,q)\in G\}.
$$

4. The cardinals $\beth_\alpha$ for $\alpha\in\mathbb{ON}$ are defined as follows: We let $\beth_0=\aleph_0$, $\beth_{\alpha+1}=2^{\beth_{\alpha}}$ and for $\alpha$ a limit we let $\beth_\alpha=\sup_{\beta<\alpha}\beth_\beta$. Since the definition of the function $\alpha\mapsto \beth_\alpha$ depends on the model of set theory $M$ in which we work, we indicate this by writing $\beth_\alpha^M$
\end{definition}

With these definitions we can state the key result of Stern that we need (\cite{stern84}, \cite[Corollary 1.8]{hjorth98}):

\begin{lemma}[Stern]\label{l.stern}
Let $M$ and $N$ be wellfounded models of ZFC, and let $(\P,p,\tau)\in M$ be a virtual Borel set of rank $\gamma$. Suppose that $\beth^M_{(-1+\gamma)}<\omega_1^N$ and that
$$
(p,p)\forces_{\P\times\P} \text{``$\tau[\dot G_l]$ codes the same Borel set as $\tau[\dot G_r]$''},
$$
where $\dot G$ is the canonical name for the generic. Then there is a Borel code $(T,f)\in N$ of rank $\gamma$ such that
$$
p\forces_{\P} \text{``$B(\emptyset,\check T,\check f)$ equals the Borel set coded by $\tau[\dot G]$''}.
$$
\end{lemma}

\subsection{The proofs of Theorem \ref{t.tfa} and \ref{t.ctblto1}.}

We will work exclusively with spaces which are finite or countable products of the spaces $2$, $\N$, $2^\N$ and $\N^\N$. We call such spaces \emph{products spaces}, and denote them by script letters $\psX$, $\psY$, $\psZ$, etc.

The proof of Theorem \ref{t.ctblto1} is a class-counting argument that uses Lemma \ref{l.stern}. However, we first need to prove that the statement ``$E\leq_B^{=_\N} F$'' is absolute for Borel equivalence relations. The next definition allow us to prove absoluteness of such statements in a slightly more general setting.

\begin{definition}
Let $F$ be a Borel equivalence relation. We will say that $F$ has the \emph{effective reductions property} (relative to $a\in\N^\N$) if whenever $E$ is a $\Delta^1_1(b)$ equivalence relation and $E\leq_B F$ then in fact $E\leq_{\Delta^1_1(a,b)} F$, i.e., there is a $\Delta^1_1(a,b)$ function $f:\N^\N\to\N^\N$ which reduces $E$ to $F$.
\end{definition}

By \cite[Proposition 2 and Theorem 15]{fofrto10}, two examples of equivalence relations with the effective reductions property includes equality on $\N$ and equality on $\N^\N$.

For a $\Delta^1_1(a)$-equivalence relation $R$ on $\psX$, let $R^+$ denote the equivalence relation on $2\times\psX$ defined by
$$
(m,x)R^+ (n,y)\iff (m=n=0)\vee (m=n=1\wedge xRy)
$$
which is also $\Delta^1_1(a)$.

\begin{lemma}\label{l.effred}
Let $E$, $F$ and $R$ be Borel equivalence relations on product spaces $\psX$, $\psY$ and $\psZ$, respectively. Suppose that $R$ is $\Delta^1_1(a_0)$ and has the effective reductions property relative to $a_0$, and that $R^+\leq_B R$. Then the statement ``$f$ is a witness to $E\leq_B^R F$'' is $\Pi^1_1(a_0)$ uniformly in the codes for $f$, $E$ and $F$ and therefore absolute. Whence, the statement ``$E\leq_B^R F$'' is $\Sigma^1_2(a_0)$ uniformly in the codes for $E$ and $F$ and is absolute.
\end{lemma}

\begin{proof}
Let $P\subseteq\N^\N\times\N\times \psX\times\psZ$ and $C\subseteq\N^\N$ be $\Pi^1_1$ such that for each $a\in\N^\N$ the family of sets
$$
P_{(a,n)}=\{(x,z)\in\psX\times\psZ: (a,n,x,z)\in P\},
$$
which is indexed by those $n\in\N$ such that $(a,n)\in C$, parametrizes the (graphs of) total $\Delta^1_1(a)$ functions. (The existence of such a set follows from \cite[2.8.2]{kanovei08}.)  Let us write $g_{(a,n)}$ for the function with graph $P_{(a,n)}$.

Fix a $\Delta^1_1(c)$ function $f$ witnessing that $E\leq_B^R F$. For $y\in\psY$, let $E_y$ be the equivalence relations defined on $\psX$ by
$$
x E_y x'\iff (f(x)Ff(x')Fy\wedge xEx')\vee (\neg (f(x)Fy)\wedge \neg(f(x')Fy))
$$
Note that $E_y$ is $\Delta^1_1$ uniformly in $y$ and the codes for $f$, $E$ and $F$, and $E_y$ depends only on $[y]_F$. Since $R$ has the effective reductions property we must have $R^+\leq_{\Delta^1_1(a_0)} R$, and saying that $f:\psX\to\psY$ is a witness to $E\leq_B^R F$ can then be expressed as:
\begin{align*}
&(\forall x,x') (xEx'\implies f(x) F f(x'))\wedge\\ 
&(\forall y\in\psY)(\exists (a,n)\in\Delta^1_1(y,a_0,b)) \text{``$g_{(a,n)}$ is a Borel reduction of $E_y$ to $R$''}.
\end{align*}
Since saying that $g_{(a,n)}$ is a Borel reduction of $E_y$ to $R$ is $\Pi^1_1$ in $(a,n)$ and the codes for $E_y$ and $R$, it follows that the statement ``$f$ is a witness to $E\leq_B^R F$'' is $\Pi^1_1(a_0)$ uniformly in the codes for $f$, $E$, $F$, for $R$ fixed as above. Thus the statement $E\leq_B^R F$ is $\Sigma^1_2(a_0)$ in the codes for $E$ and $F$.
\end{proof}

We note that taking $R$ to be $=^\N$ or $=^{\N^\N}$, the hypothesis of Lemma \ref{l.effred} is satisfied.

\begin{proof}[Proof of Theorem \ref{t.ctblto1}]
Suppose that $F$ is a $\mathbf{\Pi}^0_\gamma$ equivalence relation such that $\simeq^\alpha\leq_B^{=^\N} F$ for all $\alpha<\omega_1$. Choose $\gamma<\alpha<\omega_1$ such that $|V_\alpha|=\beth_\alpha$ (for instance, choose $\alpha$ such that $\omega+\alpha=\alpha$.) Let $\P=\Coll(\omega,\beth_\alpha)$. Let $(X_\beta)_{\beta<\beth_{\alpha+1}}$ enumerate $V_\alpha$, and let $\tau_\beta$ be the canonical name for a real coding the structure $(V_\alpha\cup\{X_\beta\},\in)$. Then we have for all $\beta,\delta<\beth_{\alpha+1}$ that
\begin{enumerate}
\item $\forces_\P \tau_\beta[\dot G]\in \check B_\alpha$
\item $\forces_{\P\times\P} \tau_\beta[\dot{G}_l] \simeq^{\alpha+2}\tau_\beta[\dot{G}_r]$
\item if $\beta\neq\delta$ then $\forces_{\P\times\P} \tau_\beta[\dot G_l] \not\simeq^{\alpha+2}\tau_\delta[\dot G_r]$
\end{enumerate}
Let $f$ be a Borel function witnessing that $\simeq^\alpha\leq_B^{=^\N} F$. For each $\beta<\beth_{\alpha+1}$, let $\sigma_\beta$ be a name for a real such that $\forces_\P f(\tau_\beta[\dot G])=\sigma_\beta[\dot G]$. Then by (2) above we have that $\forces_{\P\times\P}\sigma_\beta[\dot G_l] \check F\sigma_\beta[\dot G_r]$. Notice now that since $f$ witnesses that $\simeq^{\alpha}\leq_B^{=^\N} F$, it holds for each $\beta<\beth_{\alpha+1}$ that the set
$$
\{\delta<\beth_{\alpha+1} :\ \forces_{\P\times\P} \sigma_\beta[\dot G_l] \check F\sigma_\delta[\dot G_r]\}
$$
has size at most $\beth_\alpha$ (in the ground model.) Thus we can find distinct ordinals $\beta_\lambda<\beth_{\alpha+1}$, for $\lambda<\beth_{\alpha+1}$, such that if $\lambda_0\neq\lambda_1$ then
$$
\forces_{\P\times\P} \sigma_{\beta_{\lambda_0}}[\dot G_l] \not\!\!\check F\sigma_{\beta_{\lambda_1}}[\dot G_r].
$$
For each $\lambda<\beth_{\alpha+1}$, let $\sigma^*_\lambda$ be a $\P$-name for a Borel code for the $F$-equivalence class of $\sigma_{\beta_\lambda}$. Since $F$ is $\mathbf{\Pi}^0_\gamma$, the family $(\sigma^*_\lambda)_{\lambda<\beth_{\alpha+1}}$ is an enumeration of $\beth_{\alpha+1}$ names for $\mathbf{\Pi}^0_\gamma$ sets. It now follows from Stern's Lemma (\ref{l.stern}) that in a model $N$ in which $\beth_\gamma^M<\aleph_1^N$ and $\beth_{\alpha}^M\geq\aleph_1^N$ (which can be obtained by forcing with $\Coll(\omega,\beth_{\gamma})$, say) that there are at least $\beth_2^N=2^{2^{\aleph_0}}$ codes for distinct $\mathbf{\Pi}^0_\gamma$ sets, a contradiction.
\end{proof}

\begin{proof}[Proof of Theorem \ref{t.tfa}]
It was proven in \cite{hjorth02} that $\simeq^\alpha\leq_B\simeq^{\bf TFA}$ for all $\alpha<\omega_1$. Thus the statement follows from Theorem \ref{t.ctblto1}.
\end{proof}

\section{Conjugacy, orbit equivalence and von Neumann equivalence are not Borel}

In this section we prove Theorems \ref{t.mainthm1}, \ref{t.mainthm2}, \ref{t.mainthm1v2} and \ref{t.mainthm2v2}. Recall that ${\bf ABEL}$ denotes the set of countably infinite Abelian groups with underlying set $\N$, ${\bf TFA}\subseteq{\bf ABEL}$ the set of torsion-free abelian groups. For $A\in{\bf ABEL}$, let $\hat A$ be the character group (which is then a compact group that we equip with the Haar measure.) Fix an action $\sigma_0:\Gamma\actson\N$ and, as in \S 4, let $\sigma:\Gamma\actson A^\N$ be the generalized Bernoulli shift and $\sigma_{\hat A}$ the quotient of $\sigma$ by $\hat A$. We will need the following Lemma.

\begin{lemma}\label{l.Borel}
Let $(X,\mu)$ and $(Y,\nu)$ be a standard Borel probability spaces, and let $\Gamma$ be a countably infinite discrete group. 

(1) For any action $\sigma_0:\Gamma\actson\N$, there is a Borel function $f:{\bf ABEL}\to\act(\Gamma,X,\mu)$ such that $f(A)$ is conjugate to $\sigma_{\hat A}$, where $\sigma_{\hat A}$ is the quotient  by $\hat A$ of the generalized Bernoulli shift on $\hat A^\N$ induced by $\sigma_0$.

(2) There is a Borel function $g:\act(\Gamma,X,\mu)\times\act(\Gamma,Y,\nu)\to\act(X,\mu)$ such that $g(\rho_0,\rho_1)$ is conjugate to the action $\rho_0\times\rho_1:\Gamma\actson X\times X$. 
\end{lemma}

\begin{proof}
(1) The proof of \cite[Theorem 2 (v.2)]{tornquist11} given on pp. 343--344 in \cite{tornquist11} proves exactly this.

(2) Routine.
\end{proof}

It is practical to introduce the following notion:

\begin{definition}
We will say that $\Delta\lhd\Lambda\leq\Gamma$ is a $\{\T\}$-\emph{pseudorigid triple} if $\Gamma$ is a countable discrete group with subgroups $\Delta$ and $\Lambda$ such that $\Lambda/\Delta$ has strongly $\{\T\}$-cocycle superrigid weakly mixing malleable actions.
\end{definition}

\subsection{Conjugacy.} To sum up, proving Theorems \ref{t.mainthm1} and \ref{t.mainthm1v2} amounts to proving the following:

\begin{theorem}\label{t.conjanal}
Let $\Delta\lhd\Lambda\leq\Gamma$ be a $\{\T\}$-pseudorigid triple. Suppose $\Gamma$ admits an action $\sigma_0:\Gamma\actson\N$ such that $\sigma_0\restrict\Lambda$ is $\Delta$-suitable. Then $\simeq^{\bf TFA}\leq_B\conj^{\actfwm(\Gamma,X,\mu)}$ and $\conj^{{\actfwm(\Gamma,X,\mu)}}$ is complete analytic.
\end{theorem}

\begin{proof}
If $\Delta$ is finite then let $\sigma_{\hat A}:\Gamma\actson \hat A^\Gamma/A$ be the quotient of the Bernoulli shift of $\Gamma$ on $\hat A^\Gamma$. Then for $A,A'\in{\bf TFA}$, it follows from Theorem \ref{t.cohoquotient} that the actions $\sigma_{\hat A}\restrict\Lambda,\sigma_{\hat A'}\restrict\Lambda$ are conjugate precisely when $A\simeq A'$. Thus, if we identify $\Gamma$ and $\N$, the function $f:{\bf TFA}\to\act(\Gamma,X,\mu)$ in Lemma \ref{l.Borel}.(1) provides a Borel reduction of $\simeq^{\bf TFA}$ to $\conj^{{\actfwm(\Gamma,X,\mu)}}$.

Assume then that $\Delta$ is infinite. Fix a standard Borel probability space $(X,\mu)$ and consider $\conj^\Gamma$ in $\act(\Gamma,X,\mu)$. Let $\beta:\Gamma\actson \{0,1\}^\Gamma$ be the (usual) Bernoulli shift, where $\{0,1\}$ is equipped with the $(\frac 1 2,\frac 1 2)$ measure. By Lemma \ref{l.Borel} there is a Borel function $h:{\bf TFA}\to\act(\Gamma,X,\mu)$ such that $h(A)$ is conjugate to $\sigma_{\hat A}\times\beta$, where $\sigma_{\hat A}$ is the quotient by $\hat A$ of the generalized Bernoulli shift on $\hat A^\N$ induced by $\sigma_0:\Gamma\actson\N$. Since $\beta\restrict\Delta$ is ergodic (indeed mixing), it follows from Theorem \ref{t.pseudofamily} it holds that $A_0\simeq^{\bf TFA} A_1$ if and only if $h(A_0)\conj^{\actfwm(\Gamma,X,\mu)} h(A_1)$, thus $h$ is the required Borel reduction. 

Since $\simeq^{\bf TFA}$ was shown in \cite{domo08} to be a complete analytic equivalence relation, it follows that so is $\conj^{\actfwm(\Gamma,X\mu)}$.
\end{proof}

Let us call a subgroup $\Lambda\leq\Gamma$ \emph{almost normal} if $[\gamma\Lambda\gamma^{-1}\cap\Lambda:\Lambda]<\infty$ for all $\gamma\in\Gamma$. From the above we get:

\begin{corollary}\label{c.almostnormal}
If{\ } $\Gamma$ contains a free group $\F_n$, $n\geq 2$, as an almost normal subgroup then $\conj^{\actfwm(\Gamma,X,\mu)}$ is complete analytic.
\end{corollary}
\begin{proof}
We can assume that $n$ is so large that there is an epimorphism $\xi:\F_n\to\SL_3(\Z)$. Let $\Delta=\ker(\xi)$, and let $\sigma_0$ be the natural action of $\Gamma$ on $\Gamma/\Delta$. Note that $\Delta$ has infinite index in $\F_n$, and that the stabilizer of $\gamma\Delta$ under the action of $\Lambda$ on $\Gamma/\Delta$ is $\gamma\Delta\gamma^{-1}\cap\Lambda$, which must have infinite index in $\Lambda$ since $\Lambda$ is almost normal in $\Gamma$. Thus the action of $\Gamma$ on $\Gamma/\Delta$ is $\Delta$-suitable for $\Lambda$, and so Theorem \ref{t.conjanal} applies.
\end{proof}

\subsection{Orbit equivalence and von Neumann equivalence}

By Theorem \ref{t.tfa}, proving Theorems \ref{t.mainthm2} and \ref{t.mainthm2v2} amounts to proving the following:

\begin{theorem}\label{t.mainthm2v3}
Let $\Delta\lhd\Lambda\leq\Gamma$ be a $\{\T\}$-pseudorigid triple and suppose that $\Lambda\simeq\F_n$ for some $2\leq n\leq\infty$. Let $\Delta\lhd\Lambda$ witness this, and suppose $\Gamma$ admits an action $\sigma_0:\Gamma\actson\N$ such that $\sigma_0\restrict\Lambda$ is $\Delta$-suitable. Then $\simeq^{\bf TFA}\leq_B^{\N}\oeq^{\actfwm(\Gamma,X,\mu)}$ and $\simeq^{\bf TFA}\leq_B^{\N}\vneq^{\actfwm(\Gamma,X,\mu)}$.
\end{theorem}

The proof of requires a few additional steps of preparation. Consider the usual linear action $\rho_0:\SL_2(\Z)\actson\Z^2$, which gives rise to the measure preserving action $\rho:\SL_2(\Z)\actson \T^2$ when $\T^2$ is identified with the character group of $\Z^2$ (see e.g. \cite[\S 2]{tornquist06}.) The following easy fact is in various forms well-known.

% LEMMA

\begin{lemma}
For any point $a\in\Z^2\setminus\{0\}$ the stabilizer
$$
\stab(a)=\{g\in\SL_2(\Z): \rho_0(g)(a)=a\}
$$
is isomorphic to $\Z$.
\end{lemma}

% PROOF

\begin{proof}
We may of course assume that $a=(1,0)$. Then $\stab(a)$ consists of the upper triangular matrices with 1's on the diagonal, which is isomorphic to $\Z$.
\end{proof}

\begin{lemma}
The restriction of $\rho$ to any non-amenable subgroup $\Lambda\leq\SL_2(\Z)$ on $\T^2$ is weakly mixing.
\end{lemma}
\begin{proof}
The Koopman representation of $\SL_2(\Z)$ on $L^2_0(\T^2)$ associated to $\rho$ is isomorphic to the representation of $\SL_2(\Z)$ on $\ell^2(\Z^2\setminus\{0\})$ arising from the linear action of $\SL_2(\Z)$ on $\Z^2$. Hence it is enough to show that if $\Lambda\leq\SL_2(\Z)$ is non-amenable, then all $\Lambda$-orbits in $\Z^2\setminus\{0\}$ are infinite. This follows since the stabilizers of any $a\in\Z^2\setminus\{0\}$ is amenable and therefore have infinite index in $\Lambda$.
\end{proof}

The next Lemma follows from the analysis of the co-induced action that can be found in \cite{tornquist09} or \cite[Lemma 2.2]{ioana11}. (The reader unfamiliar with the co-induction construction should consult either of those references for background.)

\begin{lemma}\label{l.wmcoind}
Suppose $\Lambda_0\leq\Lambda\leq\Gamma$ are countable discrete subgroups, and suppose $\sigma:\Lambda\actson (X,\mu)$ is a measure preserving action such that the restriction of this action to any subgroup $\Upsilon\leq\Lambda$ which is isomorphic to a finite index subgroup in $\Lambda_0$ is weakly mixing. Then the co-induced action $\hat\sigma:\Gamma\actson X^{\Gamma/\Lambda}$ is weakly mixing when restricted to $\Lambda_0$.

In particular, if $\Lambda$ is non-amenable and the restriction of $\sigma$ to any non-amenable subgroup of $\Lambda$ is weakly mixing, then the co-induced action $\hat\sigma:\Gamma\actson X^{\Gamma/\Lambda}$ is weakly mixing when restricted to $\Lambda$.
\end{lemma}

It is well-known that $\SL_2(\Z)$ contains a subgroup (of finite index) isomorphic to $\F_2$, and therefore contains all $\F_n$, $1\leq n\leq\infty$. We let $\rho_n=\rho\restrict\F_n$, where we identify $\F_n$ with some (fixed) subgroup in $\SL_2(\Z)$. The key theorem that we need for the proof of Theorem \ref{t.mainthm2v3} is the following result due to Ioana:

\begin{theorem}[Ioana, {\cite[Theorem 1.3 and 4.7]{ioana11}}]\label{t.ioana}
Let $\Gamma$ be a countable discrete group containing a free group $\F_n$, $2\leq n\leq\infty$, and let $(X,\mu)$ be a standard Borel probability space. Let $(\sigma^i)_{i\in I}$ be a family of a.e. free ergodic m.p. $\Gamma$-actions. Suppose that $\rho_n$ is a quotient of $\sigma^i\restrict\F_n$ with a quotient map $p^i:X\to\T^2$ satisfying
$$
(\forall\gamma\in\Gamma) \mu(\{x\in X: p^i(\gamma\cdot_{\sigma^i} x)=p^i(x)\})=0.
$$
Suppose further that $\sigma^i\restrict\F_n$ is ergodic. Then:

(i) Suppose that the family $(\sigma^i)_{i\in I}$ consists of mutually orbit equivalent actions. Then
$$
\{\sigma^i\restrict\F_n: i\in I\}
$$
contains at most countably many non-conjugate actions of $\F_n$.

(ii) Suppose that the family $(\sigma^i)_{i\in I}$ consists of mutually von Neumann equivalent actions. Then
$$
\{\sigma^i\restrict\F_n: i\in I\}
$$
contains at most countably many non-conjugate actions of $\F_n$.
\end{theorem}

With these facts in hand, we can now prove Theorem \ref{t.mainthm2v3}.

\begin{proof}[Proof of Theorem \ref{t.mainthm2v3}]
Fix $\Delta\lhd\Lambda\leq\Gamma$ and $\sigma_0$ as in the statement of the theorem. As in \S 4, for $A\in{\bf ABEL}$, let $\sigma_{\hat A}$ be the quotient of the generalized Bernoulli shift of $\Gamma$ on $\hat A^\N$ induced by $\sigma_0$. Let $\hat\rho_n:(\T^2)^{\Gamma/\Lambda}$ be the coinduced action. Note that $\Delta$ must be non-amenable since otherwise it would be isomorphic to $\{1\}$, in which case $\Lambda/\Delta$ wouldn't be rigid, or else it would be isomorphic to $\Z$, in which case $\Delta$ wouldn't be normal in $\Lambda$. So by Lemma \ref{l.wmcoind} the action $\hat\rho_n\restrict\Delta$ is weakly mixing.

By Lemma \ref{l.Borel} there is a Borel map $h:{\bf TFA}\to\actfwm(\Gamma,X,\mu)$ such that $h(A)$ is conjugate to $\sigma_A\times\hat\rho_n$. Theorem \ref{t.pseudofamily} then gives that $A_0\simeq^{\bf TFA} A_1$ if and only if $h(A_0)\conj h(A_1)$, making $h$ a Borel reduction of $\simeq^{\bf TFA}$ to $\conj^{\actfwm(\Gamma,X,\mu)}$.

The actions $\sigma_A\times\hat\rho_n$ clearly satisfy Ioana's Theorem \ref{t.ioana} (taking $p$ to be the projection onto the the $\Lambda$-coordinate of $(\T^2)^{\Gamma/\Lambda}$), and therefore at most countably many non-conjugate actions of the form $\sigma_A\times\hat\rho_n$ can be orbit equivalent or von Neumann equivalent. Thus $h$ is a countable-to-1 Borel reduction of $\simeq^{\bf TFA}$ to $\oeq^{\actfwm(\Gamma,X,\mu)}$ and $\vneq^{\actfwm(\Gamma,X,\mu)}$, and so by Theorem \ref{t.tfa} the relations $\oeq^{\actfwm(\Gamma,X,\mu)}$ and $\vneq^{\actfwm(\Gamma,X,\mu)}$ are not Borel.
\end{proof}

\begin{corollary}
If $\Gamma$ contains a free group $\F_n$, $n\geq 2$, as an almost normal subgroup then $\oeq^{\actfwm(\Gamma,X,\mu)}$ and $\vneq^{\actfwm(\Gamma,X,\mu)}$ are analytic and not Borel.
\end{corollary}

\subsection{Isomorphism of factors is not Borel.}

The map that associates to an element in $\act(\Gamma,X,\mu)$ the corresponding group measure space von Neumann algebra was shown to be Borel in \cite{sato09}, when the set of von Neumann algebras on $L^2(\Gamma\times X)$ is given the Effros Borel structure. It is therefore clear from the above that isomorphism of $\II_1$ factors is not Borel. In \cite{sato09} it was shown that the isomorphism relation of $\II_1$ factors is even \emph{complete analytic}. However, no such conclusion about type $\II_\infty$ and type $\III_\lambda$, $0\leq\lambda\leq 1$ could be drawn from the construction found there. The actions considered in this paper do allow conclusions to be drawn about the Borel complexity of isomorphism of factors of type $\II_\infty$ and $\III_\lambda$. Specifically, we have:

\begin{theorem}\label{t.vnalgebras}
The isomorphism relation for factors of type $\II_\infty$ and type $\III_\lambda$ for each $0\leq\lambda\leq 1$ is analytic but not Borel.
\end{theorem}
\begin{proof}
The proof follows exactly the same idea as in \cite[\S 3]{sato09}. Let $H=\ell^2(\N)$ be the separable infinite dimensional complex Hilbert space. Let $\Gamma=\Lambda=\F_n$, where $n$ is so large that there is $\Delta\lhd\F_n$ such that $\F_n/\Delta\simeq\SL_3(\Z)$. Let $h:{\bf TFA}\to\actfwm(\Gamma,X,\mu)$ be the Borel map considered in the proof of Theorem \ref{t.mainthm2v3}. Then it follows from  \cite[Lemma 5]{sato09} that there is a Borel map $\tilde h:{\bf TFA}\to\vN(H)$ (where $\vN(H)$ denotes the standard Borel space of von Neumann algebras acting on $H$, see \cite{effros65} or \cite{sato09} for background) such that $\tilde h(A)\simeq L^\infty(X,\mu)\rtimes_{h(A)}\F_n$ for all  $A\in{\bf TFA}$.

Note now that by \cite[Theorem 4.3]{ioana11} (see also the discussion before \cite[Corollary 4.6]{ioana11}), the factors $\tilde h(A)$ are $\mathcal{HT}$-factors. Therefore the argument used to prove \cite[Theorem 9]{sato09} applies. To be specific, the argument there shows that the map $\tilde h_{\II_\infty}:{\bf TFA}\to \vN(H\otimes H): A\mapsto \tilde h(A)\otimes \mathcal B(\mathcal H)$ is a Borel countable-to-1 reduction of $\simeq^{\bf TFA}$ to isomorphism of $\II_\infty$ factors since $\tilde h(A)\otimes\mathcal B(H)\simeq \tilde h(A')\otimes\mathcal B(H)$ if and only if $\tilde h(A)\simeq\tilde h(A')$. Thus isomorphism of type $\II_\infty$ factors is analytic and not Borel by Theorem \ref{t.tfa}.

Furthermore, it also follows from the proof of \cite[Theorem 9]{sato09} that if $N$ is an injective factor of type $\III_\lambda$, $0\leq\lambda\leq 1$, then $\tilde h(A)\otimes N\simeq\tilde h(A')\otimes N$ if and only if $\tilde h(A)\simeq \tilde h(A')$. Thus the map $\tilde h_{\III_\lambda}: {\bf TFA}\to\vN(H\otimes H): A\mapsto \tilde h(A)\otimes N$ is a Borel countable-to-1 reduction of isomorphism of in ${\bf TFA}$ to isomorphism in the class of separably acting $\III_\lambda$ factors. Thus isomorphism of type $\III_\lambda$ factors is not Borel.
\end{proof}

\begin{remark}\label{r.sdf}
Following a presentation of the results of this paper by the second author at the XI Atelier International de Th\'eorie des Ensembles in Luminy, France, Sy Friedman made the following remark: Under the extra set-theoretic assumption that ``$0^\sharp$ exists'' (see e.g. \cite{kanamori09}) or, which is more crude, that there exists a measurable cardinal, any analytic non-Borel set is automatically complete analytic. Therefore we obtain stronger conclusions in Theorems \ref{t.mainthm2v3} and \ref{t.vnalgebras} under this extra set-theoretic assumption.

However, we feel that it is merely a limitation of our proof that the stronger conclusion is not reached above within standard set theory (ZFC.) In fact we conjecture that the following holds:

\begin{conjecture}
Let $\Gamma$ be any non-amenable countably infinite discrete group. Then conjugacy, orbit equivalence and von Neumann equivalence in $\actfwm(\Gamma,X,\mu)$ are complete analytic equivalence relations (as subsets of $\actfwm(\Gamma,X,\mu)^2)$.
\end{conjecture}

\end{remark}

\bibliographystyle{amsplain}
\bibliography{vneqanalytic}

\bigskip

{\smaller \sc
\noindent  Department of Mathematics, California Institute of Technology\\
Pasadena, CA 91125

}

{\smaller \tt \noindent epstein@caltech.edu}

\bigskip

\smallskip

{\smaller\sc
\noindent Department of Mathematics, University of Copenhagen\\
Universitetsparken 5, 2100 Copenhagen, Denmark
}

{\smaller \tt \noindent asgert@math.ku.dk}

\end{document}